\documentclass[12pt]{amsart}
\usepackage{txfonts}      %{article} was 12pt latex e
\usepackage{amssymb}
\usepackage{eucal}
\usepackage{amsmath}
\usepackage{amscd}
\usepackage{xcolor}
\usepackage{multicol}
\usepackage[all]{xy}           %xypic macro for latex2.09
\usepackage{graphicx}
\usepackage{color}
\usepackage{colordvi}
\usepackage{xspace}
\usepackage{tikz}
\usepackage{makecell}
\usepackage{appendix}
\usepackage{amsthm}

\usepackage{ifpdf}
\ifpdf
\usepackage[colorlinks,final,backref=page,hyperindex]{hyperref}
\else
\usepackage[colorlinks,final,backref=page,hyperindex,hypertex]{hyperref}
\fi

\usepackage[active]{srcltx} %SRC Specials for DVI Searching

\begin{document}

%%%%%%%%%%%%%%%%%%%%%%%% Statements

\newtheorem{thm}{Theorem}[section]
\newtheorem{lem}[thm]{Lemma}
\newtheorem{cor}[thm]{Corollary}
\newtheorem{pro}[thm]{Proposition}
\theoremstyle{definition}
\newtheorem{defi}[thm]{Definition}
\newtheorem{ex}[thm]{Example}
\newtheorem{rmk}[thm]{Remark}
\newtheorem{pdef}[thm]{Proposition-Definition}
\newtheorem{condition}[thm]{Condition}

\renewcommand{\labelenumi}{{\rm(\alph{enumi})}}
\renewcommand{\theenumi}{\alph{enumi}}

\newcommand {\emptycomment}[1]{} %to remove paragraphs

\newcommand{\nc}{\newcommand}
\newcommand{\delete}[1]{}

\nc{\tred}[1]{\textcolor{red}{#1}}
\nc{\tblue}[1]{\textcolor{blue}{#1}}
\nc{\tgreen}[1]{\textcolor{green}{#1}}
\nc{\tpurple}[1]{\textcolor{purple}{#1}}
\nc{\tgray}[1]{\textcolor{gray}{#1}}
\nc{\torg}[1]{\textcolor{orange}{#1}}
\nc{\tmag}[1]{\textcolor{magenta}}
\nc{\btred}[1]{\textcolor{red}{\bf #1}}
\nc{\btblue}[1]{\textcolor{blue}{\bf #1}}
\nc{\btgreen}[1]{\textcolor{green}{\bf #1}}
\nc{\btpurple}[1]{\textcolor{purple}{\bf #1}}

%\nc{\revise}[1]{\textcolor{blue}{#1}}

%%%%%%%% new symbols

\nc{\tforall}{\ \ \text{for all }}
\nc{\hatot}{\,\widehat{\otimes} \,}
\nc{\complete}{completed\xspace}
\nc{\wdhat}[1]{\widehat{#1}}

\nc{\ts}{\mathfrak{p}}
\nc{\mts}{c_{(i)}\ot d_{(j)}}

\nc{\NA}{{\bf NA}}
\nc{\LA}{{\bf Lie}}
\nc{\CLA}{{\bf CLA}}

\nc{\cybe}{CYBE\xspace}
\nc{\nybe}{NYBE\xspace}
\nc{\ccybe}{CCYBE\xspace}

\nc{\ndend}{pre-Novikov\xspace}
\nc{\calb}{\mathcal{B}}
\nc{\rk}{\mathrm{r}}
\newcommand{\g}{\mathfrak g}
\newcommand{\h}{\mathfrak h}
\newcommand{\pf}{\noindent{$Proof$.}\ }
\newcommand{\frkg}{\mathfrak g}
\newcommand{\frkh}{\mathfrak h}
\newcommand{\Id}{\rm{Id}}
\newcommand{\gl}{\mathfrak {gl}}
\newcommand{\ad}{\mathrm{ad}}
\newcommand{\add}{\frka\frkd}
\newcommand{\frka}{\mathfrak a}
\newcommand{\frkb}{\mathfrak b}
\newcommand{\frkc}{\mathfrak c}
\newcommand{\frkd}{\mathfrak d}
\newcommand {\comment}[1]{{\marginpar{*}\scriptsize\textbf{Comments:} #1}}

\nc{\vspa}{\vspace{-.1cm}}
\nc{\vspb}{\vspace{-.2cm}}
\nc{\vspc}{\vspace{-.3cm}}
\nc{\vspd}{\vspace{-.4cm}}
\nc{\vspe}{\vspace{-.5cm}}

%%%%%%%%%%%%%%%%%%%%%%% old symbols

\nc{\disp}[1]{\displaystyle{#1}}
\nc{\bin}[2]{ (_{\stackrel{\scs{#1}}{\scs{#2}}})}  %binomial coeff
\nc{\binc}[2]{ \left (\!\! \begin{array}{c} \scs{#1}\\
    \scs{#2} \end{array}\!\! \right )}  %binomial coeff
\nc{\bincc}[2]{  \left ( {\scs{#1} \atop
    \vspace{-.5cm}\scs{#2}} \right )}  %binomial coeff
\nc{\ot}{\otimes}
\nc{\sot}{{\scriptstyle{\ot}}}
\nc{\otm}{\overline{\ot}}
\nc{\ola}[1]{\stackrel{#1}{\la}}%${\Bbb Z}$

\nc{\scs}[1]{\scriptstyle{#1}} \nc{\mrm}[1]{{\rm #1}}

\nc{\dirlim}{\displaystyle{\lim_{\longrightarrow}}\,}
\nc{\invlim}{\displaystyle{\lim_{\longleftarrow}}\,}

\nc{\bfk}{{\bf k}} \nc{\bfone}{{\bf 1}}
\nc{\rpr}{\circ}
%\nc{\apr}{\cdot}
\nc{\dpr}{{\tiny\diamond}}
\nc{\rprpm}{{\rpr}}

%%%%%%%%%%%%%%%%%%%%% roman fonts, in alphabetic order
\nc{\mmbox}[1]{\mbox{\ #1\ }} \nc{\ann}{\mrm{ann}}
\nc{\Aut}{\mrm{Aut}} \nc{\can}{\mrm{can}}
\nc{\twoalg}{{two-sided algebra}\xspace}
\nc{\colim}{\mrm{colim}}
\nc{\Cont}{\mrm{Cont}} \nc{\rchar}{\mrm{char}}
\nc{\cok}{\mrm{coker}} \nc{\dtf}{{R-{\rm tf}}} \nc{\dtor}{{R-{\rm
tor}}}
\renewcommand{\det}{\mrm{det}}
\nc{\depth}{{\mrm d}}
\nc{\End}{\mrm{End}} \nc{\Ext}{\mrm{Ext}}
\nc{\Fil}{\mrm{Fil}} \nc{\Frob}{\mrm{Frob}} \nc{\Gal}{\mrm{Gal}}
\nc{\GL}{\mrm{GL}} \nc{\Hom}{\mrm{Hom}} \nc{\hsr}{\mrm{H}}
\nc{\hpol}{\mrm{HP}}  \nc{\id}{\mrm{id}} \nc{\im}{\mrm{im}}

\nc{\incl}{\mrm{incl}} \nc{\length}{\mrm{length}}
\nc{\LR}{\mrm{LR}} \nc{\mchar}{\rm char} \nc{\NC}{\mrm{NC}}
\nc{\mpart}{\mrm{part}} \nc{\pl}{\mrm{PL}}
\nc{\ql}{{\QQ_\ell}} \nc{\qp}{{\QQ_p}}
\nc{\rank}{\mrm{rank}} \nc{\rba}{\rm{RBA }} \nc{\rbas}{\rm{RBAs }}
\nc{\rbpl}{\mrm{RBPL}}
\nc{\rbw}{\rm{RBW }} \nc{\rbws}{\rm{RBWs }} \nc{\rcot}{\mrm{cot}}
\nc{\rest}{\rm{controlled}\xspace}
\nc{\rdef}{\mrm{def}} \nc{\rdiv}{{\rm div}} \nc{\rtf}{{\rm tf}}
\nc{\rtor}{{\rm tor}} \nc{\res}{\mrm{res}} \nc{\SL}{\mrm{SL}}
\nc{\Spec}{\mrm{Spec}} \nc{\tor}{\mrm{tor}} \nc{\Tr}{\mrm{Tr}}
\nc{\mtr}{\mrm{sk}}

%%%%%%%%%%%%%%%%%% bold face
\nc{\ab}{\mathbf{Ab}} \nc{\Alg}{\mathbf{Alg}}

%%%%%%%%%%%%%%%%%%%Bbb fonts
\nc{\BA}{{\mathbb A}} \nc{\CC}{{\mathbb C}} \nc{\DD}{{\mathbb D}}
\nc{\EE}{{\mathbb E}} \nc{\FF}{{\mathbb F}} \nc{\GG}{{\mathbb G}}
\nc{\HH}{{\mathbb H}} \nc{\LL}{{\mathbb L}} \nc{\NN}{{\mathbb N}}
\nc{\QQ}{{\mathbb Q}} \nc{\RR}{{\mathbb R}} \nc{\BS}{{\mathbb{S}}} \nc{\TT}{{\mathbb T}}
\nc{\VV}{{\mathbb V}} \nc{\ZZ}{{\mathbb Z}}

%%%%%%%%%%%%%%%%%%% cal fonts

\nc{\calao}{{\mathcal A}} \nc{\cala}{{\mathcal A}}
\nc{\calc}{{\mathcal C}} \nc{\cald}{{\mathcal D}}
\nc{\cale}{{\mathcal E}} \nc{\calf}{{\mathcal F}}
\nc{\calfr}{{{\mathcal F}^{\,r}}} \nc{\calfo}{{\mathcal F}^0}
\nc{\calfro}{{\mathcal F}^{\,r,0}} \nc{\oF}{\overline{F}}
\nc{\calg}{{\mathcal G}} \nc{\calh}{{\mathcal H}}
\nc{\cali}{{\mathcal I}} \nc{\calj}{{\mathcal J}}
\nc{\call}{{\mathcal L}} \nc{\calm}{{\mathcal M}}
\nc{\caln}{{\mathcal N}} \nc{\calo}{{\mathcal O}}
\nc{\calp}{{\mathcal P}} \nc{\calq}{{\mathcal Q}} \nc{\calr}{{\mathcal R}}
\nc{\calt}{{\mathcal T}} \nc{\caltr}{{\mathcal T}^{\,r}}
\nc{\calu}{{\mathcal U}} \nc{\calv}{{\mathcal V}}
\nc{\calw}{{\mathcal W}} \nc{\calx}{{\mathcal X}}
\nc{\CA}{\mathcal{A}}

%%%%%%%%%%%%%%%%%%  frak fonts
\nc{\fraka}{{\mathfrak a}} \nc{\frakB}{{\mathfrak B}}
\nc{\frakb}{{\mathfrak b}} \nc{\frakd}{{\mathfrak d}}
\nc{\oD}{\overline{D}}
\nc{\frakF}{{\mathfrak F}} \nc{\frakg}{{\mathfrak g}}
\nc{\frakm}{{\mathfrak m}} \nc{\frakM}{{\mathfrak M}}
\nc{\frakMo}{{\mathfrak M}^0} \nc{\frakp}{{\mathfrak p}}
\nc{\frakS}{{\mathfrak S}} \nc{\frakSo}{{\mathfrak S}^0}
\nc{\fraks}{{\mathfrak s}} \nc{\os}{\overline{\fraks}}
\nc{\frakT}{{\mathfrak T}}
\nc{\oT}{\overline{T}}
%\nc{\frakx}{{\mathfrak x}}
\nc{\frakX}{{\mathfrak X}} \nc{\frakXo}{{\mathfrak X}^0}
\nc{\frakx}{{\mathbf x}}
%\nc{\frakTxo}{{\frakTx}^0}
\nc{\frakTx}{\frakT}      %All rooted trees, correspond to \ncsha(X)
\nc{\frakTa}{\frakT^a}        % rooted trees for \ncsha(A)
\nc{\frakTxo}{\frakTx^0}   % rooted trees for \ncshao(X)
\nc{\caltao}{\calt^{a,0}}   % rooted trees for \ncshao(A)
\nc{\ox}{\overline{\frakx}} \nc{\fraky}{{\mathfrak y}}
\nc{\frakz}{{\mathfrak z}} \nc{\oX}{\overline{X}}

%%%%%%%%%%%%%%%%%%%%%%%%%%%%%%%%%%%%%%%%%%%%%%%%%%%%%%%%%%%%%%%%%%

\title[Lie bialgebras via Novikov bialgebras]{Infinite-dimensional Lie bialgebras\\ via affinization of Novikov bialgebras and Koszul duality}

\author{Yanyong Hong}
\address{School of Mathematics, Hangzhou Normal University,
Hangzhou 311121, PR China}
\email{yyhong@hznu.edu.cn}

\author{Chengming Bai}
\address{Chern Institute of Mathematics \& LPMC, Nankai University, Tianjin 300071, PR China}
\email{baicm@nankai.edu.cn}

\author{Li Guo}
\address{Department of Mathematics and Computer Science, Rutgers University, Newark, NJ 07102, USA}
         \email{liguo@rutgers.edu}



\begin{abstract}
Balinsky and Novikov showed that the affinization of a Novikov
algebra naturally defines a Lie algebra, a property that in fact
characterizes the Novikov algebra. It is also an instance of the
operadic Koszul duality. In this paper, we develop a bialgebra
theory for the Novikov algebra, namely the Novikov bialgebra,
which is characterized by the fact that its affinization (by a
quadratic right Novikov algebra) gives an infinite-dimensional Lie bialgebra, suggesting a Koszul duality for properads.
A Novikov bialgebra is also characterized as a Manin triple of Novikov algebras.
The notion of Novikov Yang-Baxter equation is introduced, whose
skewsymmetric solutions can be used to produce Novikov bialgebras and hence Lie bialgebras.
Moreover, these solutions also give rise to skewsymmetric solutions of the classical Yang-Baxter equation in the infinite-dimensional Lie algebras from the Novikov algebras.
\end{abstract}

\maketitle

\vspace{-1.2cm}

\tableofcontents

\vspace{-1.2cm}

\allowdisplaybreaks

\section{Introduction}

In this paper, the Balinsky-Novikov construction~\cite{BN} of infinite-dimensional Lie algebras via the affinization of Novikov algebras is lifted to a construction of infinite-dimensional Lie bialgebras, built on the notions of Novikov bialgebras and affinization in the bialgebra context.

\subsection{Infinite-dimensional Lie algebras by affinization}

\subsubsection{Lie algebras via affinization}
An important construction of infinite-dimensional Lie algebras is a process of affinization (without centers) which equips a Lie algebra structure on the tensor product, over a field $\bfk$ of characteristic zero,
 \begin{equation} \label{eq:aff}
\hat{A}:=A[t,t^{-1}]:=A\ot \bfk[t,t^{-1}]
\end{equation}
of a certain algebra $A$ of finite dimension with another algebraic structure on the space of Laurent polynomials.
We broadly call this process the construction of infinite-dimensional Lie algebras via the {\bf $\calp$-algebra affinization}, if $A$ is a $\calp$-algebra for an algebraic operad $\calp$.

The classical instance of the affinization is the Lie algebra affinization $\hat{L}=L\ot \bfk[t,t^{-1}]$ where $L$ is a finite-dimensional Lie algebra and $\bfk[t,t^{-1}]$ is the usual commutative associative algebra.

\subsubsection{Novikov algebras and their affinization}

Another instance of affinization is the Novikov algebra affinization.
Novikov algebras were introduced in connection with Hamiltonian
operators in the formal variational calculus~\cite{GD1, GD2} and
Poisson brackets of hydrodynamic type~\cite{BN}. As an important
subclass of pre-Lie algebras, they have attracted great interest because of their broad connections, in particular to Lie conformal algebras~\cite{X1} and vertex algebras~\cite{BLP}.

The Novikov algebra also shows it significance in its role in affinization.

\begin{thm}\label{correspond1} $($Balinsky-Novikov~\cite{BN}$)$
    Let $A$ be a vector space with a binary operation $\circ$. Define a binary operation on $A[t,t^{-1}]:=A\otimes \bfk[t,t^{-1}]$ by
    \begin{eqnarray}\label{infLie1} \notag
        [at^i, bt^j]=i(a\circ b)t^{i+j-1}-j(b\circ a)t^{i+j-1}\;\;\tforall
        a,b\in A, i, j\in\mathbb{Z},
    \end{eqnarray}
    where $a t^i:=a\otimes t^{i}$. Then $(A[t,t^{-1}], [\cdot,\cdot])$
    is a Lie algebra if and only if $(A, \circ)$ is a Novikov algebra.
\end{thm}

Note that, thanks to the only if part of the theorem, the Novikov algebra is characterized by this affinization, further highlighting the importance of the Novikov algebra.

The Novikov algebra affinization gives many of the
infinite-dimensional Lie algebras which are important in
mathematical physics, such as the Witt algebra, centerless
Heisenberg-Virasoro algebra  \cite{ADKP} and
Schr\"odinger-Virasoro algebra  \cite{H}. There are also close
relationships between these infinite-dimensional Lie algebras and
the corresponding Novikov algebras~\cite{BN, PB}.

\subsubsection{Operadic Koszul duality} \label{sss:opdalg}

Note that, unlike the commutative associative algebra structure on $\bfk[t,t^{-1}]$ in the Lie algebra affinization,  in the Novikov algebra affinization, the space $\bfk[t,t^{-1}]$ is equipped with the right Novikov algebra product
$t^i\diamond t^j:=it^{i+j}$ (see
Example~\ref{Laurent-right Nov}(\ref{it:novex2})).
As it turns out, these two instances of affinization constructions of infinite-dimensional Lie algebras fit into the Ginzburg-Kapranov operadic Koszul duality~\cite{GK}.

Let $\mathcal{P}$ be a binary quadratic operad and $\mathcal{P}^{\textup{!`}}$ be its operadic Koszul dual. By~\cite[Corollary~2.29]{GK} (see also~\cite[Proposition~7.6.5]{LV}), the tensor product of a $\mathcal{P}$-algebra with a $\mathcal{P}^{\textup{!`}}$-algebra is naturally endowed with a Lie algebra structure, giving a potentially very general procedure of constructing Lie algebras.
Thus Lie algebra affinization and Novikov algebra affinization have the operadic interpretation that the Koszul dual of the operad of Lie algebras (resp. Novikov algebras) is the operad of commutative associative algebras (resp. right Novikov algebras~\cite{Dz}).

\subsection{Infinite-dimensional Lie bialgebras}\

A Lie bialgebra is composed of a Lie algebra and a Lie coalgebra
joined together by a cocycle condition. Introduced by V. Drinfel'd
in his study of Hamiltonian mechanics and Poisson-Lie
groups~\cite{Dr}, the Lie bialgebra is the
classical limit of a quantized universal enveloping
algebra~\cite{CP,Dr1}. The great importance of Lie bialgebras is
also reflected by its close relationship with several other fundamental
notions. In the finite-dimensional case, Lie bialgebras are
characterized by Manin triples of Lie algebras. Moreover,
solutions of the classical Yang-Baxter equation (CYBE), also
called the classical $r$-matrices, naturally give rise to
coboundary Lie bialgebras. Furthermore, such solutions are
provided by $\calo$-operators of Lie algebras, which in turn are provided by
pre-Lie algebras~\cite{Bai,CP,Ku1,STS}.
These connections are depicted in the diagram
\begin{equation}
    \begin{split}
        \xymatrix{
            \text{pre-Lie}\atop\text{ algebras} \ar[r]     & \mathcal{O}\text{-operators of}\atop\text{Lie algebras}\ar[r] &
            \text{solutions of}\atop \text{CYBE} \ar[r]  & \text{Lie}\atop \text{ bialgebras}  \ar@{<->}[r] &
        \text{Manin triples of} \atop \text{Lie algebras} }
    \end{split}
    \label{eq:bigliediag}
\end{equation}

Infinite-dimensional Lie bialgebras have attracted a lot of
attention since the very beginning of the Lie bialgebra study.
Solutions of the CYBE with spectral parameters provide
infinite-dimensional Lie bialgebras. In particular, the classical
structures of two infinite-dimensional quantum groups, namely the
Yangians and quantum affine algebras, are the Lie bialgebras on
the loop algebras  and affine Lie algebras, corresponding to
rational and trigonometric solutions
respectively~\cite{CP}.

Numerous studies of Lie bialgebra structures on infinite-dimensional Lie algebras have been carried out, including for the Witt algebra, Virasoro algebra, Schr\"odinger-Virasoro algebra and current algebras~\cite{AMSZ,HLS,KPSST,Mi,MSZ1,NT}.
Note that for each of these Lie bialgebras, the dual of the Lie
coalgebra is not of the same type as the Lie algebra. For example,
the dual of the Lie coalgebra in the bialgebra structure on the
Witt algebra is not the Witt algebra~\cite{NT}. Thus it is natural
to investigate Lie bialgebras in which the Lie algebras and the
dualized Lie algebras from the Lie coalgebras are of the same
type.

\subsection{Novikov bialgebras and their affinization construction of Lie bialgebras}\

Our goal is to construct Lie bialgebras by applying the
affinization process as in Eq.~\eqref{eq:aff} to both the Lie
algebras and Lie coalgebras, focusing on the Novikov algebra
affinization. More precisely, we expand Theorem~\ref{correspond1}
to the context of bialgebras so that, in the resulting Lie
bialgebras, in addition to the Lie algebras, the dualized Lie
algebras from the Lie coalgebras are the affinizations of Novikov
algebras. To achieve this goal, we need to overcome several
challenges.

\subsubsection{Affinization of Novikov coalgebras}
As a first step, we consider the dual version of the Novikov algebra affinization, for Novikov coalgebras. As noted just above, this ``Novikov coalgebra affinization" cannot be achieved by the usual notion of Lie coalgebras.
The reason behind this is that the product on the space of Laurent polynomials for the Novikov algebra affinization could not come from dualizing a usual coproduct (see Lemma~\ref{l:laucop}).

To move forward, we introduce the completed tensor
product to serve as the target space of more general coproducts,
leading to the notion of completed Lie coalgebras and other related concepts. Especially,
there is a completed right Novikov coalgebra structure on the
space of Laurent polynomials, allowing us to carry out the
affinization process. Then we obtain a dual version of
Theorem~\ref{correspond1}, that is, there is a natural completed
Lie coalgebra structure on the tensor product of a Novikov
coalgebra and the completed right Novikov coalgebra of Laurent
polynomials (Theorem~\ref{Constr-Lie coalgebra}), a property that in fact characterizes the Novikov coalgebra.

\subsubsection{Novikov bialgebras and their affinization}
We next lift the affinization characterization of the Novikov algebra in Theorem~\ref{correspond1} and its aforementioned coalgebra variation to the level of bialgebras. For this purpose, we introduce, on the one hand, the notion of a Novikov bialgebra, composed of a Novikov algebra and a Novikov coalgebra satisfying suitable compatibility conditions and, on the other hand, the notion of a quadratic $\ZZ$-graded right Novikov algebra, as a $\ZZ$-graded right Novikov algebra equipped with an invariant bilinear form.
Note that the latter algebra is not the $\ZZ$-graded right Novikov
bialgebra obtained from the Novikov bialgebra by taking the opposite
operations. Then we show that the tensor product of a
finite-dimensional Novikov bialgebra and a quadratic $\ZZ$-graded
right Novikov algebra can be naturally endowed with a completed
Lie bialgebra (Theorem \ref{Gcorrespond6}). The converse of this
result also holds when the quadratic $\ZZ$-graded right
Novikov algebra is taken from the space of Laurent polynomials,
giving the desired characterization of the Novikov bialgebra that
its affinization is a Lie bialgebra.

\subsubsection{Operadic interpretation}
In view of the operadic interpretation of
Theorem~\ref{correspond1} by the Ginzburg-Kapranov duality
applied to the operads of the Novikov algebra and the right Novikov algebra (see Section~\ref{sss:opdalg}),
Theorem~\ref{Gcorrespond6} should also have an operadic
interpretation in terms of a Koszul duality of properads (or
dioperads), that can be applied to the properads of Novikov
bialgebras and quadratic right Novikov algebras. Such a duality
has been established by Gan and Vallette for certain quadratic
dioperads and properads~\cite{Ga,Va}. Theorem~\ref{Gcorrespond6}
gives a strong motivation to generalize the duality beyond the quadratic case.

As in the case of Koszul duality of operads, from a Koszul duality of properads, one might expect that a Lie bialgebra structure can be obtained on the tensor product of a $\calp$-algebra with a  $\calp^{\textup{!`}}$-algebra for a suitable properad $\calp$.
This would provide a very general procedure to construct Lie bialgebras, of which the Novikov bialgebra affinization would be a special case.

\subsubsection{Manin triple characterizations}
There is also a characterization of finite-dimensional Novikov bialgebras by Manin triples of Novikov algebras which is comparable to the Manin triple characterization of Lie bialgebras. This characterization depends critically on the condition that the linear dual of every Novikov algebra can be equipped with a representation which can be expressed as a nonzero linear combination of the left and right multiplication operators of the Novikov algebra.  We establish this condition and thereby disprove a claim made in~\cite{Ku} (see
Remark~\ref{rk:kuper}).

The Manin triple characterization of Novikov bialgebras
also offers a natural way to see why quadratic right
Novikov algebras need to appear in the above Lie bialgebras
construction via Novikov bialgebra affinization. Indeed, the Manin
triple characterization of Novikov bialgebras should be compatible
with the Manin triple characterization of the Lie bialgebras
obtained by Novikov bialgebras, in the sense that the tensor product of a Manin triple of Novikov algebras and a right Novikov algebra equipped with a suitable extra structure should naturally give a Manin triple of Lie algebras. We find that this extra structure is precisely the quadratic property on the right Novikov algebra (see Remark~\ref{rk:quadnov}).

\subsubsection{Novikov Yang-Baxter equation and its affinization}
Extending the close relationship between the CYBE and Lie
bialgebras, we define an analogy of the CYBE for Novikov algebras,
called the Novikov Yang-Baxter equation (NYBE), so that skewsymmetric solutions of the NYBE in Novikov algebras give Novikov bialgebras.

On the other hand, from solutions of the NYBE in a Novikov algebra, there is a
construction of solutions of the CYBE in the Lie algebra from the Novikov algebra affinization (Proposition~\ref{pro:NYBE}).
Hence the latter solutions can be regarded as an
affinization of the former solutions.

We introduce the $\mathcal O$-operator of a Novikov algebra as an operator form of the NYBE. It gives a skewsymmetric solution of the NYBE in a semidirect product
Novikov algebra. Moreover, a natural example of $\mathcal
O$-operators of Novikov algebras is given by \ndend algebras whose operad is the splitting of the operad of Novikov
algebras~\cite{BBGN}. An example of \ndend algebras is given by
Zinbiel algebras equipped with derivations.

Overall, the diagram~\eqref{eq:bigliediag} for Lie bialgebras is extended to a diagram for Novikov bialgebras. Furthermore, the two diagrams are related as illustrated in the following commutative diagram.
\begin{equation}
    \begin{split}
        \xymatrix{
\text{\ndend}\atop\text{ algebras} \ar[r]  \ar[d]      &
\mathcal{O}\text{-operators of}\atop\text{Novikov
algebras}\ar[r] \ar[d]&
            \text{solutions of}\atop \text{NYBE} \ar[r]  \ar[d]& \text{Novikov}\atop \text{ bialgebras}   \ar[d] \ar@{<->}[r] & \text{Manin triples of} \atop \text{Novikov algebras} \ar[d]\\
\text{pre-Lie}\atop\text{ algebras} \ar[r] &
\mathcal{O}\text{-operators of}\atop\text{Lie
algebras}\ar[r]&
            \text{solutions of}\atop \text{CYBE} \ar[r]
            & \text{Lie}\atop \text{ bialgebras} \ar@{<->}[r] & \text{Manin triples of} \atop \text{Lie algebras}  }
    \end{split}
    \label{eq:bigdiag}
\end{equation}
The commutativity of the two right squares are given in Proposition~\ref{eq:bialgdouble} and Corollary~\ref{cor:same} respectively and those of the other ones can be similarly obtained. In particular this provides a large supply of Lie bialgebras.

Furthermore, we introduce the notion of a quasi-Frobenius Novikov algebra to characterize a class of skewsymmetric solutions of the NYBE.
In addition, we use quasi-Frobenius Novikov algebras to construct quasi-Frobenius $\ZZ$-graded Lie algebras, thereby providing many examples of infinite-dimensional quasi-Frobenius Lie algebras.
\subsection{Outline of the paper}
The paper is briefly outlined as follows.

In Section~\ref{sec:liebialg}, we first introduce the notions of Novikov bialgebras and quadratic $\ZZ$-graded right Novikov algebras, as well as completed Lie bialgebras.
We show that there is a natural completed Lie bialgebra structure on the tensor product of a finite-dimensional Novikov bialgebra and a quadratic $\ZZ$-graded right Novikov  algebra.
In the special case when the quadratic
$\ZZ$-graded right  Novikov algebra is on the space of Laurent
polynomials, we obtain a characterization of the Novikov bialgebra by the affinization.

In Section~\ref{sec:nybe}, Manin triples of Novikov algebras are introduced to give an equivalent description of finite-dimensional Novikov bialgebras by means of matched pairs. We also present a natural construction of Manin triples of Lie algebras from Manin triples of Novikov algebras and a given quadratic right Novikov algebra.
A special class of Novikov bialgebras defined by two-tensors leads to the notion of  the NYBE in such a way that a skewsymmetric solution of the NYBE gives a Novikov bialgebra. The notions of an $\mathcal{O}$-operator of a Novikov algebra, a pre-Novikov algebra and a quasi-Frobenius Novikov algebra are also introduced to interpret and construct solutions of the NYBE.

In Section~\ref{sec:novlie}, for a Novikov algebra $A$ and a
quadratic $\ZZ$-graded right Novikov algebra $B$, we present a
natural construction of skewsymmetric solutions of the CYBE in the Lie algebra $A\otimes B$ from the skewsymmetric
solutions of the NYBE in $A$. Moreover, a construction of quasi-Frobenius $\ZZ$-graded Lie algebras from quasi-Frobenius Novikov algebras corresponding to a class of skewsymmetric solutions of the NYBE is given.
In particular, utilizing
Section~\ref{sec:nybe} leads to a construction of completed Lie
bialgebras and quasi-Frobenius $\ZZ$-grade Lie algebras from pre-Novikov algebras. The construction is further illustrated by an example.

\section{Infinite-dimensional Lie bialgebras from Novikov bialgebras by Koszul duality}
\label{sec:liebialg}

In this section, the Balinsky-Novikov theorem on Lie algebraic
structures via Novikov algebra affinization is extended to Lie
bialgebras, built on the notions of Novikov bialgebras and
quadratic $\ZZ$-graded right Novikov algebras. The notion of
Novikov bialgebras is introduced in Section~\ref{ss:novbialg}.
Section~\ref{ss:rightnovcoalg} presents the notion of
completed right Novikov coalgebras together with a coalgebra version of the Balinsky-Novikov theorem.  Section~\ref{ss:novbiaff} constructs Lie bialgebras from Novikov bialgebras and quadratic $\ZZ$-graded right Novikov
algebras.

\subsection{Novikov bialgebras}
\label{ss:novbialg}
We start with the classical notions of left and right Novikov algebras.
\begin{defi}
 A {\bf Novikov algebra} is a vector space $A$ with a binary
    operation $\circ$ satisfying
\begin{eqnarray}
        \label{lef} \notag
        (a\circ b)\circ c-a\circ (b\circ c)&=&(b\circ a)\circ c-b\circ (a\circ c),\\
        \label{Nov} \notag
        (a\circ b)\circ c&=&(a\circ c)\circ b ~~~~\tforall  a, b,c\in A.
\end{eqnarray}
    A {\bf right  Novikov algebra} is a vector space $A$ with a binary
 operation $\diamond $ satisfying
    \begin{eqnarray}
        &&\label{RNA1} \notag
        (a\diamond b)\diamond c-a\diamond (b\diamond c)=(a\diamond c)\diamond b-a\diamond (c\diamond b),\\
        &&\label{RNA2} \notag
        a\diamond (b\diamond c)=b\diamond (a\diamond c) ~~~~\tforall  a, b,c\in A.
    \end{eqnarray}
\end{defi}

\begin{rmk}
    \begin{enumerate}
        \item It is obvious that $(A,\circ)$ is a Novikov algebra if and only if its opposite $(A,\diamond)$ defined by $a\diamond b\coloneqq b\circ a$ for
        all $a,b\in A$ is a right Novikov algebra. Thus a Novikov algebra should be more precisely called a left Novikov algebra. We will follow the usual notion and still call it a Novikov algebra unless an emphasis is needed.
        \item The Koszul dual of the operad of (left) Novikov algebras  is the operad of right Novikov algebras~\cite{Dz}.
    \end{enumerate}
\end{rmk}

\begin{ex}\label{Laurent-right Nov}
    \begin{enumerate}
        \item \label{ConstrNov}
The classic example of a Novikov algebra was given by S.~Gelfand~\cite{GD1}.
Let $(A,\cdot)$ be a commutative associative algebra and $D$ be a derivation. Then the binary operation
        \begin{equation}
            a \circ b\coloneqq a\cdot D(b)~~~~\tforall a, b\in A,
        \end{equation}
        defines a Novikov algebra $(A,\circ)$. The binary operation
        $$a\diamond b\coloneqq D(a) \cdot b~~~~\tforall a, b\in A$$
        defines a right Novikov algebra $(A,\diamond)$.
        \item \label{it:novex2}
        As a special case,
        equip the Laurent polynomial algebra ${\bf k}[t,t^{-1}]$ with the natural derivation $D\coloneqq \frac{d}{dt}$. Then $(B={\bf k}[t,t^{-1}], \diamond)$ is a right Novikov algebra with $\diamond$ given by
        \begin{eqnarray}
            t^i\diamond t^j\coloneqq D(t^i)t^j=it^{i+j-1}~~~~\tforall i, j\in \mathbb{Z}.
        \end{eqnarray}
    \end{enumerate}
\end{ex}

Let $A$ be a vector space. Let
$$\tau:A\otimes A\rightarrow A\otimes A,\quad a\otimes b\mapsto b\otimes a \tforall a,b\in A,$$
be the flip operator. Dualizing the notion of a Novikov algebra, we give the next notion.
\begin{defi}
    A {\bf(left) Novikov coalgebra} is a vector space $A$ with a linear map $\Delta: A\rightarrow A\otimes A$, called the coproduct, such that
\begin{eqnarray}
\label{Lc3}(\id\otimes \Delta)\Delta(a)-(\tau\otimes {\rm id})(\id\otimes \Delta)\Delta(a)&=&(\Delta\otimes \id)\Delta(a)-(\tau\otimes {\rm id})(\Delta\otimes \id)\Delta(a),\\
\label{Lc4}(\tau\otimes {\rm id})(\id\otimes \Delta)\tau
     \Delta(a)&=&(\Delta\otimes \id)\Delta(a) \tforall a\in A.
\end{eqnarray}
\end{defi}

Let $A$ be a finite-dimensional vector space and
$\Delta:A\rightarrow A\otimes A$ be a coproduct.
Let $\cdot:A^*\ot A^*\to A^*$ be the
corresponding binary operation on the dual space. As in the case of associative algebras and coalgebras, the pair $(A,\Delta)$ is a Novikov coalgebra if and only if $(A^*,\cdot)$ is a Novikov algebra.

Let $(A,\circ)$ be a Novikov algebra. Define another binary operation $\star$ on $A$ by
\begin{eqnarray} \notag
    a\star b=a\circ b+b\circ a\quad \tforall ~~a,~~b\in A.
\end{eqnarray}
Let $L_A$, $R_A:
A\rightarrow {\rm End}_{\bf k}(A)$ be the linear maps defined respectively by
\begin{eqnarray*}
    L_A(a)(b)\coloneqq a\circ b,\quad R_A(a)(b)\coloneqq b\circ a\quad \tforall  a,~~b\in A.
\end{eqnarray*}
Define $L_{A,\star}: A\rightarrow {\rm End}_{\bf k}(A)$ by $L_{A,\star}=L_A+R_A$.

Now we introduce the main notion in our study.

\begin{defi}\label{Def-Novbi}
A {\bf Novikov bialgebra} is a tripe $(A,\circ,\Delta)$ where  $(A,\circ)$ is a Novikov algebra and $(A, \Delta)$ is a Novikov coalgebra such that, for all $a$, $b\in A$,
the following conditions are satisfied.
 {\small\begin{eqnarray}
            &\label{Lb5}\Delta(a\circ b)=(R_A (b)\otimes \id)\Delta (a)+(\id\otimes L_{A,\star}(a))(\Delta(b)+\tau\Delta(b)),&\\
            &\label{Lb6}(L_{A,\star}(a)\otimes \id)\Delta(b)-(\id\otimes L_{A,\star}(a))\tau\Delta (b)=(L_{A,\star}(b)\otimes \id)\Delta(a)-(\id\otimes L_{A,\star}(b))\tau\Delta(a),&\\
            &\label{Lb7}(\id\otimes R_A(a)-R_A(a)\otimes
            \id)(\Delta(b)+\tau\Delta(b))=(\id\otimes R_A(b)-R_A(b)\otimes
            \id)(\Delta(a)+\tau\Delta(a)).&
    \end{eqnarray}}
\end{defi}

\begin{ex}\label{ex-Nov-bi}
    Let $(A,\circ)$ be the $2$-dimensional Novikov algebra
    in~\cite{BaiMeng} with a basis $\{e_1,e_2\}$ whose multiplication is
    given by
    \begin{eqnarray*}
        e_1\circ e_1=e_1,~~~e_2\circ e_1=e_2,~~~e_1\circ e_2=e_2\circ e_2=0.
    \end{eqnarray*}
    Define $\Delta_A: A\rightarrow A\otimes A$ by
    \begin{eqnarray*}
        \Delta_A(e_1)=\lambda e_2\otimes e_2,\quad \Delta_A(e_2)=0,
    \end{eqnarray*}
    for a fixed $\lambda\in \bf k$. Then it is direct to verify that $(A, \circ, \Delta_A)$ is a Novikov bialgebra.
\end{ex}

\subsection{Completed right Novikov coalgebras and \complete Lie coalgebras}
\label{ss:rightnovcoalg}

For the rest of this section, we assume that $A$ is a finite-dimensional vector space.

\begin{defi}
A {\bf $\ZZ$-graded right Novikov algebra} (resp. a {\bf $\ZZ$-graded Lie algebra})  is a right Novikov algebra $(B,\diamond )$ (resp. a Lie algebra $(B,[\cdot,\cdot]$)) with a linear decomposition
    $B=\oplus_{i\in \ZZ} B_i$ such that each $B_i$ is finite-dimensional and $B_i \diamond B_j\subset B_{i+j}$ (resp. $[B_i, B_j]\subset B_{i+j}$) for all $i, j\in \ZZ$.
\end{defi}

More generally, the definition still makes sense without the finite-dimensional condition on $B_i$. We impose the condition for the application in this paper.
\begin{ex}\label{Laurent Novikov}
    For the Laurent polynomial algebra in Example~\ref{Laurent-right Nov} \eqref{it:novex2}, let $B_i={\bf k}t^{i+1}$ for $i\in \mathbb{Z}$. Then
    $(B=\oplus_{i\in \mathbb{Z}}B_i, \diamond)$ is a $\ZZ$-graded
    right Novikov algebra.
\end{ex}

Generalizing Theorem~\ref{correspond1}, the pairing of a (left) Novikov algebra and a $\ZZ$-graded right Novikov algebra gives a $\ZZ$-graded Lie algebra as stated below.

\begin{thm} \label{pp:tensorlie} Let $(A,\circ)$ be a (left)
    Novikov algebra and $(B,\diamond)$ be a $\ZZ$-graded right Novikov algebra. Define
    a binary operation on $A\otimes B$ by
    \begin{equation}\label{Liecons1} \notag
        [a_1\otimes b_1, a_2\otimes b_2]=a_1\circ a_2\otimes b_1\diamond
        b_2-a_2\circ a_1\otimes b_2\diamond b_1~~~~\tforall  a_1, a_2\in A,
        b_1, b_2\in B.
    \end{equation}
    Then $(A\otimes B, [\cdot, \cdot])$ is a $\ZZ$-graded Lie algebra,
    called {\bf the induced Lie algebra} (from $(A, \circ)$ and $(B,
        \diamond))$. Further, if $(B, \diamond)=( {\bf k}[t,t^{-1}], \diamond)$ is
    the $\mathbb{Z}$-graded right Novikov algebra in Example~\ref{Laurent Novikov}, then $(A\otimes B, [\cdot, \cdot])$ is a
    $\mathbb Z$-graded Lie algebra if and only if $(A,\circ)$
    is a Novikov algebra.
\end{thm}

\begin{proof}
    By \cite[Coro.~2.2.9]{GK} (also see \cite[Prop.~7.6.5]{LV}), $(A\otimes B, [\cdot, \cdot])$ is a Lie algebra. Moreover, since $(B, \diamond)$ is $\ZZ$-graded, $(A\otimes B, [\cdot, \cdot])$ is a $\ZZ$-graded Lie algebra. When  $(B, \diamond)=( {\bf k}[t,t^{-1}], \diamond)$, this result follows from~Theorem \ref{correspond1} (that is
    \cite[Lemma~1]{BN}).
\end{proof}

Our main goal is to extend the Novikov algebra affinization to Novikov coalgebras and further to Novikov bialgebras. For the Novikov coalgebra affinization, we first need to find a coalgebra structure on the space of Laurent polynomials whose graded linear dual is the right Novikov algebra of Laurent polynomials in  Example
\ref{Laurent Novikov}. As the next simple result shows, this cannot be achieved by a usual coproduct.

\begin{lem} \label{l:laucop} Let $C=\bfk[x,x^{-1}]$ be the $\ZZ$-graded (by degree) space of Laurent polynomials and let $A=\bfk[t,t^{-1}]$ be the graded linear dual.
Then the right Novikov algebra product $\diamond$ on $A$ from
Example~\ref{Laurent Novikov} cannot be the induced product from the graded linear dual of any coproduct $\delta:C\to C\ot C$.
\end{lem}
\begin{proof}
Suppose that such a coproduct $\delta$ exists. Denote
$\delta(x^i)=\sum\limits_{p,q\in \ZZ} c_{p,q} x^p\ot x^q$. Then
the graded linear duality
$\langle x^i,t^j\rangle=\delta_{i,j}$ gives the duality $
\langle \delta(x^i),t^j\ot t^k\rangle=\langle x^i,t^j\diamond t^k\rangle,$ yielding
$c_{j,k}=\delta_{i-j-k+1,0}j$. Thus
$$ \delta(x^i)=\sum_{i-p-q+1=0} \delta_{i-p-q+1,0}p x^p\ot x^q
= \sum_{p\in \ZZ} px^p\ot x^{i-p+1}.$$
This is an infinite sum and so cannot be defined by $\delta:C\to C\ot C$.
\end{proof}

Thus to carry out the Novikov coalgebra affinization, we need to extend the codomain of the coproduct $\delta:C\to C\ot C$ to allow infinite sums, in the spirit of topological coalgebras~\cite{Ta}.

Let $C=\oplus_{i\in \ZZ} C_i$ and $D=\oplus_{j\in \ZZ} D_j$ be $\ZZ$-graded vector spaces. We call the {\bf completed tensor product} of $C$ and $D$ to be the vector space
$$C\hatot  D\coloneqq \prod_{i,j\in\ZZ} C_i\ot D_j.
$$
If $C$ and $D$ are finite-dimensional, then $C\hatot  D$ is just the usual tensor product $C\otimes D$.

\begin{rmk}
We note that $C\hatot  D$ is the completion of $C\ot D$ with respect to the topology defined by the filtration $\Fil^n (C\ot D)\coloneqq \oplus_{|i|\geq n} C_i\ot D_j$~\cite{AMSZ,Bou,Ta}. For this reason, we call the above structure completed. We do not need further restrictions on the tensor product or topological coproduct in a topological vector space. In particular, we do not require that the coproducts $\delta$ (resp. $\Delta$) defined below be continuous.
\end{rmk}

In general, an element in $C\hatot  D$ is an infinite formal sum $\sum_{i,j\in\ZZ} \ts_{ij} $ with $\ts_{ij}\in C_i\ot D_j$. So $\ts_{ij}=\sum_\alpha c_{i\alpha}\ot d_{j\alpha}$ for pure tensors $c_{i\alpha}\ot d_{j\alpha}\in C_i\ot D_j$ with $\alpha$ in a finite index set.
Thus a general term of $C\hatot  D$ is a possibly infinite sum
\begin{equation}\label{eq:ssum}
 \sum_{i,j,\alpha} c_{i\alpha}\ot d_{j\alpha},
 \end{equation}
where $i,j\in \ZZ$ and $\alpha$ is in a finite index set (which might depend on $i,j$).

With these notations, for linear maps $f:C\to C'$ and $g:D\to D'$, define
$$ f\hatot g: C\hatot D \to C' \hatot D',
\quad \sum_{i,j,\alpha} c_{i\alpha}\ot d_{j\alpha}\mapsto  \sum_{i,j,\alpha} f(c_{i\alpha})\ot g(d_{j\alpha}).
$$
Also the twist map $\tau$ has its completion
$$ \widehat{\tau}: C\hatot C \to C\hatot C, \quad
\sum_{i,j,\alpha} c_{i\alpha} \ot d_{j\alpha} \mapsto \sum_{i,j,\alpha} d_{j\alpha}\ot c_{i\alpha}.
$$
Finally, we define a (completed) coproduct to be a linear map
$$ \Delta: C\to C\hatot C, \quad \Delta(a)
:=\sum_{i,j,\alpha} a_{1i\alpha}\ot a_{2j\alpha}.
$$
Then we have the well-defined map
$$ (\Delta \hatot \id)\Delta(a)=(\Delta \hatot \id)\Big(\sum_{i,j,\alpha} a_{1i\alpha}\ot a_{2j\alpha}\Big) :=\sum_{i,j,\alpha} \Delta(a_{1i\alpha})\ot a_{2j\alpha}
\in C\hatot C \hatot C.
$$

\begin{defi}
\begin{enumerate}
\item   A {\bf \complete right Novikov coalgebra} is a pair $(C,\Delta)$ where $C=\oplus_{i\in \ZZ}C_i$ is a $\mathbb{Z}$-graded vector space and
    $\Delta: C\rightarrow C\hatot  C$ is a linear map satisfying
    \begin{eqnarray}
        &&\label{top-Nov-Coalg-1}(\Delta \hatot  \id)\Delta(a)-(\id \hatot  \widehat{\tau})(\Delta \hatot  \id)\Delta(a)=(\id \hatot  \Delta)\Delta(a)-(\id \hatot  \widehat{\tau})(\id \hatot  \Delta )\Delta(a),\\
        &&\label{top-Nov-Coalg-2}(\id \hatot  \Delta)\Delta(a)=(\widehat{\tau}\hatot  \id) (\id \hatot  \Delta)\Delta(a)\;\;\;\tforall  a\in C.
    \end{eqnarray}
\item
A {\bf \complete Lie coalgebra}
 is a pair $(L, \delta)$, where
$L=\oplus_{i\in \ZZ} L_i$
  is a $\ZZ$-graded vector space and $\delta: L\rightarrow L\hatot L$ is a linear map satisfying
    \begin{eqnarray}
        \label{lia1} \notag
        &&\delta(a)=-\widehat{\tau} \delta(a), \\
        \label{lia2} \notag
        &&(\id\hatot \delta)\delta(a)-(\widehat{\tau}\hatot  \id)(\id\hatot  \delta)\delta(a)=(\delta\hatot  \id)\delta(a) ~~~\tforall  a\in L.
    \end{eqnarray}
\end{enumerate}
\end{defi}

In the above two definitions, when $C$ (resp. $L$)  is finite-dimensional, $(C, \Delta)$ (resp. $(L, \delta)$) is just a right Novikov coalgebra as the opposite of a Novikov coalgebra (resp. a Lie coalgebra).

\begin{ex}\label{Laurent-coproduct}
On the vector space $B\coloneqq {\bf k}[t,t^{-1}]$, define a linear map $\Delta_B: B\rightarrow B \hatot  B$ by
\begin{eqnarray}\label{eq:laurentco}
     \notag
\Delta_B(t^j)=\sum_{i\in \mathbb{Z}}(i+1)t^{-i-2}\otimes  t^{j+i} ~~~~\tforall  j\in \mathbb{Z}.
\end{eqnarray}
One can directly check that $(B={\bf k}[t,t^{-1}], \Delta_B)$ is a \complete right Novikov coalgebra.
\end{ex}

Now we give the dual version of Theorem~\ref{pp:tensorlie}.
\begin{thm}\label{Constr-Lie coalgebra}
    Let $(A,\Delta_A)$ be a Novikov coalgebra, $(B, \Delta_B)$ be a \complete right Novikov coalgebra and $L:=A\otimes B$.
     Define the linear map $\delta:
    L\rightarrow L \hatot  L$ by
    \begin{eqnarray}\label{GLiebi1}
        \delta(a\otimes b)=(\id_{L\widehat{\ot} L}-\widehat{\tau})\big(\Delta_A(a)\bullet \Delta_B(b)\big) \tforall a\in A, b\in B.
    \end{eqnarray}
Here for $\Delta_A(a)=\sum_{(a)}
a_{(1)}\otimes a_{(2)}$ in the Sweedler notation and $\Delta_B(b)
=\sum _{i,j,\alpha} b_{1i\alpha}\ot b_{2j\alpha}$ as in
Eq.~\eqref{eq:ssum}, we set
\begin{equation} \label{eq:multcoprod} \notag
\Delta_A(a)\bullet \Delta_B(b)\coloneqq \sum_{(a)}\sum_{i,j,\alpha} (a_{(1)}\otimes b_{1i\alpha})\otimes (a_{(2)}\otimes b_{2j\alpha}).
\end{equation}
Then $(L, \delta)$ is a \complete Lie coalgebra. Furthermore, if $(B={\bf k}[t,t^{-1}], \Delta_B)$ is the \complete right Novikov coalgebra given in Example \ref{Laurent-coproduct}, then $(L=A\otimes B, \delta)$ is a \complete Lie coalgebra if and only if $(A, \Delta_A)$ is a Novikov coalgebra.
\end{thm}

\begin{proof}
   Obviously, $\delta=-\widehat{\tau }\,\delta$. For
    $\sum_{\ell}a'_\ell\otimes a''_\ell\otimes a'''_\ell\in A\otimes A\otimes A$ and
    $\sum_{i,j,k,\alpha}b'_{i\alpha}\otimes b''_{j\alpha}\otimes b'''_{k\alpha}\in B\hatot  B\hatot  B$, we denote
    \begin{eqnarray*}
    \Big(\sum_{\ell}a'_\ell\otimes a''_\ell\otimes a'''_\ell\Big)\bullet \Big(\sum_{i,j,k,\alpha }b'_{i\alpha}\otimes b''_{j\alpha}\otimes b'''_{k\alpha}\Big):=\sum_\ell \sum_{i,j,k,\alpha}(a'_\ell\otimes b'_{i\alpha})\otimes (a''_\ell\otimes b''_{j\alpha})\otimes (a'''_\ell\otimes  b'''_{k\alpha}).
    \end{eqnarray*}
Let $a\otimes b\in A\otimes B$. Applying the definition of $\delta$ and using the above notation, we obtain
    \small{\begin{eqnarray*}
        &&\Big((\id \hatot\delta)\delta-(\widehat{\tau}\hatot  \id)(\id \hatot \delta)\delta-(\delta \hatot  \id)\delta\Big)(a\otimes b)\\
        &=& \Big((\id \otimes \Delta_A)\Delta_A(a)\Big)\bullet \Big((\id \hatot  \Delta_B)\Delta_B(b)\Big)-\Big((\id\otimes \tau)(\id \otimes \Delta_A)\Delta_A(a)\Big) \bullet \Big((\id\hatot  \widehat{\tau})(\id \hatot  \Delta_B)\Delta_B(b)\Big) \\
        &&-\Big((\id \otimes \Delta_A)\tau \Delta_A(a)\Big) \bullet \Big((\id \hatot \Delta_B)\widehat{\tau} \Delta_B(b)\Big)
        +\Big((\id \otimes \tau)(\id \otimes \Delta_A)\tau \Delta_A(a)\Big) \bullet \Big((\id \hatot  \widehat{\tau})(\id \hatot \Delta_B)\widehat{\tau }\Delta_B(b)\Big)\\
        &&-\Big((\tau\otimes \id)(\id \otimes \Delta_A)\Delta_A(a)\Big)\bullet \Big((\widehat{\tau}\hatot  \id)(\id \hatot  \Delta_B)\Delta_B(b)\Big)\\
        &&+\Big((\tau\otimes \id)(\id\otimes \tau)(\id \otimes \Delta_A)\Delta_A(a) \Big)
        \bullet \Big((\widehat{\tau}\hatot  \id)(\id\hatot  \widehat{\tau})(\id \hatot  \Delta_B)\Delta_B(b)\Big)\\
        &&+\Big((\tau\otimes \id)(\id \otimes \Delta_A)\tau \Delta_A(a)\Big) \bullet \Big((\widehat{\tau}\hatot  \id)(\id \hatot \Delta_B)\widehat{\tau} \Delta_B(b)\Big)\\
        &&-\Big((\tau\otimes \id)(\id \otimes \tau)(\id \otimes \Delta_A)\tau \Delta_A(a)\Big) \bullet \Big((\widehat{\tau}\hatot  \id)(\id \hatot  \widehat{\tau})(\id \hatot \Delta_B)\widehat{\tau }\Delta_B(b)\Big)\\
        &&-\Big((\Delta_A\otimes \id)\Delta_A(a)\Big)\bullet \Big((\Delta_B\hatot  \id)\Delta_B(b)\Big)+\Big((\tau\otimes \id)(\Delta_A\otimes \id)\Delta_A(a)\Big) \bullet \Big((\tau\hatot \id)(\Delta_B\hatot  \id)\Delta_B(b)\Big)\\
        &&+\Big((\Delta_A\otimes \id)\tau \Delta_A(a)\Big) \bullet \Big((\Delta_B\hatot  \id)\tau \Delta_B(b)\Big)-\Big((\tau \otimes \id)(\Delta_A\otimes \id)\tau \Delta_A(a)\Big) \bullet \Big((\tau \hatot \id)(\Delta_B\hatot  \id)\tau \Delta_B(b)\Big).
    \end{eqnarray*}}
    By Eqs. (\ref{top-Nov-Coalg-1}) and (\ref{top-Nov-Coalg-2}), we obtain
\begin{eqnarray*}
        &&(\Delta_B\hatot  \id)\widehat{\tau} \Delta_B(b)=(\id \otimes \widehat{\tau})    (\id \hatot  \Delta_B)\Delta_B(b), \;\;\;(\widehat{\tau}\hatot  \id)(\Delta_B\hatot  \id)\Delta_B(b)=(\widehat{\tau}\hatot  \id)(\id\hatot  \widehat{\tau})(\id \hatot  \Delta_B)\Delta_B(b),\\
        &&(\id \hatot  \widehat{\tau})(\id\widehat{ \otimes} \Delta_B)\widehat{\tau} \Delta_B(b)=(\widehat{\tau}\hatot  \id)(\id\hatot  \widehat{\tau})(\id\hatot \Delta_B)\Delta_B(b)-(\Delta_B\hatot \id)\widehat{\tau} \Delta_B(b)+(\widehat{\tau}\hatot  \id)(\Delta_B\hatot  \id)\widehat{\tau} \Delta_B(b).
    \end{eqnarray*}
Then applying these equalities together with Eqs. (\ref{top-Nov-Coalg-1}) and (\ref{top-Nov-Coalg-2}), we get
\small{\begin{eqnarray*}
        &&\Big((\id \hatot \delta)\delta-(\widehat{\tau}\hatot  \id)(\id \hatot \delta)\delta-(\delta \hatot  \id)\delta\Big)(a\otimes b)\\
        &=&\!\!\! \Big((\id \otimes \Delta_A)\Delta_A(a)\Big)\bullet \Big( (\id \hatot  \Delta_B)\Delta_B(b)\Big)-\Big((\id\otimes \tau)(\id \otimes \Delta_A)\Delta_A(a)\Big) \bullet  \Big( (\id\hatot  \widehat{\tau})(\id \hatot  \Delta_B)\Delta_B(b) \Big)\\
        &&\!\!\!\!-\Big((\id \otimes \Delta_A)\tau \Delta_A(a)\Big) \bullet \Big((\id \hatot \Delta_B)\widehat{\tau} \Delta_B(b)\Big)
        +\Big((\id \otimes \tau)(\id \otimes \Delta_A)\tau \Delta_A(a)\Big) \bullet \Big((\widehat{\tau}\hatot  \id)(\id\hatot  \widehat{\tau})(\id\hatot \Delta_B)\Delta_B(b)\\
        &&-(\Delta_B\hatot \id)\widehat{\tau} \Delta_B(b)+(\widehat{\tau}\hatot  \id)(\Delta_B\hatot  \id)\widehat{\tau} \Delta_B(b)\Big)
        -\Big((\tau\otimes \id)(\id \otimes \Delta_A)\Delta_A(a)\Big)\bullet \Big((\widehat{\tau}\hatot  \id)(\id \hatot  \Delta_B)\Delta_B(b)\Big)\\
        &&+\Big((\tau\otimes \id)(\id\otimes \tau)(\id \otimes \Delta_A)\Delta_A(a)\Big) \bullet \Big((\widehat{\tau}\hatot  \id)(\id\hatot  \widehat{\tau})(\id \hatot  \Delta_B)\Delta_B(b)\Big)\\
        &&+\Big((\tau\otimes \id)(\id \otimes \Delta_A)\tau \Delta_A(a) \Big)\bullet \Big( (\widehat{\tau}\hatot  \id)(\id \hatot \Delta_B)\widehat{\tau} \Delta_B(b)\Big)\\
        &&-\Big((\tau\otimes \id)(\id \otimes \tau)(\id \otimes \Delta_A)\tau \Delta_A(a) \Big )\bullet \Big((\widehat{\tau}\hatot  \id)(\id \hatot  \widehat{\tau})(\id \hatot \Delta_B)\widehat{\tau }\Delta_B(b)\Big)\\
        &&-\Big((\Delta_A\otimes \id)\Delta_A(a)\Big)\bullet \Big((\id \hatot  \widehat{\tau})(\Delta_B\hatot  \id)\Delta_B(a)
        +(\id\hatot  \Delta_B)\Delta_B(b)-(\id \hatot \widehat{\tau})(\id\widehat{ \otimes} \Delta_B)\Delta_B(b)\Big)\\
        &&+\Big((\tau\otimes \id)(\Delta_A\otimes \id)\Delta_A(a)\Big) \bullet \Big((\widehat{\tau}\hatot  \id)(\id \hatot  \widehat{\tau})(\Delta_B\hatot  \id)\Delta_B(a)
        +(\widehat{\tau}\hatot  \id)(\id\hatot \Delta_B)\Delta_B(b)\\
        &&\hspace{4.5cm} -(\widehat{\tau}\hatot  \id)(\id \hatot \widehat{\tau})(\id \hatot \Delta_B)\Delta_B(b)\Big)\\
        &&+\Big((\Delta_A\otimes \id)\tau \Delta_A(a) \Big)\bullet \Big((\Delta_B\hatot  \id)\widehat{\tau} \Delta_B(b)\Big)
        -\Big((\tau \otimes \id)(\Delta_A\otimes \id)\tau \Delta_A(a)\Big) \bullet \Big( (\widehat{\tau} \hatot \id)(\Delta_B\hatot  \id)\widehat{\tau} \Delta_B(b)\Big)\\
        &=& \Big((\id \otimes \Delta_A)\Delta_A(a)-(\tau\otimes \id)(\Delta_A\otimes \id)\Delta_A(a)-(\Delta_A\otimes \id)\Delta_A(a)+(\tau\otimes \id)(\Delta_A\otimes \id)\Delta_A(a)\Big)\\
        &&\bullet \Big((\id\widehat{ \otimes} \Delta_B)\Delta_B(b)\Big)\\
        &&+\Big((\id \otimes \tau)(\id \otimes \Delta_A)\Delta_A(a)-(\id \otimes \tau)(\id \otimes \Delta_A)\tau \Delta_A(a)-(\Delta_A\otimes \id)\Delta_A(a)
        +(\Delta_A\otimes \id)\tau \Delta_A(a)\Big) \\
        &&\ \ \bullet \Big((\id \hatot \widehat{ \tau})(\id\widehat{ \otimes} \Delta_B)\Delta_B(b)\Big)\\
        &&-\Big((\id\otimes \Delta_A)\tau \Delta_A(a)-(\tau\otimes \id)(\Delta_A\otimes \id)\Delta_A(a)\Big)
        \bullet \Big((\id \hatot  \Delta_B)\widehat{\tau }\Delta_B(b)\Big)\\
        &&+\Big((\id \otimes \tau)(\id \otimes \Delta_A)\tau \Delta_A(a)-(\id \otimes \tau)(\tau\otimes \id)(\Delta_A\otimes \id)\Delta_A(a)\Big)
        \bullet \Big((\id \hatot \widehat{\tau})(\id \hatot \Delta_B)\widehat{\tau} \Delta_B(b)\Big)\\
        &&+\Big((\tau\otimes \id)(\id\otimes \Delta_A)\tau \Delta_A(a)-(\Delta_A\otimes \id)\Delta_A(a)\Big)\bullet \Big( (\widehat{\tau}\hatot  \id)(\id\hatot \Delta_B)\widehat{\tau }\Delta_B(b)\Big)\\
        &&+\Big((\id\otimes \tau)(\id\otimes \Delta_A)\tau \Delta_A(a)
        +(\tau\otimes \id)(\id\otimes \tau)(\id \otimes \Delta_A)\Delta_A(a)\\
        &&\ \ \ -(\tau\otimes \id)(\Delta_A\otimes \id)\Delta_A(a)
         -(\tau\otimes \id)(\Delta_A\otimes \id)\tau \Delta_A(a)\Big)\bullet \Big((\widehat{\tau}\hatot  \id)(\Delta_B\hatot  \id)\widehat{\tau} \Delta_B(b)\Big)\\
        &=&0.
    \end{eqnarray*}}
Therefore, $(L, \delta)$ is a \complete Lie coalgebra.

If $(B={\bf k}[t,t^{-1}], \Delta_B)$ is the \complete right
Novikov coalgebra given in Example \ref{Laurent-coproduct}, then
the corresponding $\delta$ is given by
 \begin{eqnarray}
\label{Lie-coproduct} \delta(at^k)=\sum_{i\in \mathbb{Z}}\sum_{(a)}(i+1)(a_{(1)}t^{-i-2}\otimes a_{(2)}t^{k+i}-a_{(2)}t^{k+i}\otimes a_{(1)}t^{-i-2}) ~~~~\tforall  a\in A, k\in \mathbb{Z}.
\end{eqnarray}

Suppose that $(\frakg, \delta)$ is a \complete Lie coalgebra. Then we have
\begin{eqnarray*}
0&=&(\id\hatot  \delta)\delta(at^k)-(\widehat{\tau}\widehat{\otimes} \id)(\id\hatot  \delta)\delta(at^k)-(\delta\hatot  \id)\delta(at^k)\\
&=&\sum_{i, j\in \mathbb{Z}}\sum_{(a)} \Big((i+1)(j+1){a_{(1)}}t^{-i-2}\otimes ({a_{(21)}}t^{-j-2}\otimes {a_{(22)}}t^{k+i+j}-{a_{(22)}}t^{k+i+j}\otimes {a_{(21)}}t^{-j-2})\\
&&-(i+1)(j+1)({a_{(2)}}t^{k+i}\otimes ({a_{(11)}}t^{-j-2}\otimes {a_{(12)}}t^{-i-2+j}- {a_{(12)}}t^{-i-2+j}\otimes {a_{(11)}}t^{-j-2}))\\
&&-(i+1)(j+1)({a_{(21)}}t^{-j-2}\otimes {a_{(1)}}t^{-i-2} \otimes {a_{(22)}}t^{k+i+j}-{a_{(22)}}t^{k+i+j}\otimes {a_{(1)}}t^{-i-2}\otimes  {a_{(21)}}t^{-j-2})\\
&&+(i+1)(j+1)({a_{(11)}}t^{-j-2}\otimes {a_{(2)}}t^{k+i}\otimes {a_{(12)}}t^{-i-2+j}-{a_{(12)}}t^{-i-2+j}\otimes {a_{(2)}}t^{k+i}\otimes  {a_{(11)}}t^{-j-2})\\
&&-(i+1)(j+1)({a_{(11)}}t^{-j-2}\otimes {a_{(12)}}t^{-i+j-2}-{a_{(12)}}t^{-i-2+j}\otimes {a_{(11)}}t^{-j-2})\otimes {a_{(2)}}t^{k+i}\\
&&+(i+1)(j+1)({a_{(21)}}t^{-j-2}\otimes {a_{(22)}}t^{k+i+j}-{a_{(22)}}t^{k+i+j}\otimes {a_{(21)}}t^{-j-2})\otimes {a_{(1)}}t^{-i-2}\Big).
\end{eqnarray*}
Let $k=2$. Comparing the coefficients of $t^{-1}\otimes t^{-1}\otimes 1$, we obtain
\begin{eqnarray*}
0=\sum_{(a)}(a_{(2)}\otimes a_{(12)}\otimes a_{(11)}-a_{(12)}\otimes a_{(2)}\otimes a_{(11)})= \tau_{13}((\tau\otimes \id)(\id \otimes \Delta_A)\tau \Delta_A(a)-(\Delta_A\otimes\id)\Delta_A(a)),
\end{eqnarray*}
where $\tau_{13}(a_1\otimes a_2\otimes a_3)=a_3\otimes a_2\otimes a_1$ for all $a_1$, $a_2$, $a_3\in A$.
Similarly, let $k=0$. Comparing the coefficients of $t^{-1}\otimes t^{-3}\otimes 1$,
we obtain
\begin{eqnarray*}
0&=&\sum_{(a)}(-a_{(22)}\otimes a_{(1)}\otimes a_{(21)}-a_{(12)}\otimes a_{(2)}\otimes a_{(11)}
+a_{(12)}\otimes a_{(11)}\otimes a_{(2)}+a_{(22)}\otimes a_{(21)}\otimes a_{(1)})\\
&=&\tau_{13}((\id\otimes \Delta_A)\Delta_A(a)-(\tau\otimes \id)(\id\otimes \Delta_A)\Delta_A(a)
-(\Delta_A\otimes \id)\Delta_A(a)+(\tau\otimes \id)(\Delta_A\otimes \id)\Delta_A(a))\\
&&+(\tau\otimes \id+\tau_{13})((\Delta_A\otimes\id)\Delta_A(a)-(\tau\otimes \id)(\id\otimes \Delta_A)\tau \Delta_A(a)).
\end{eqnarray*}
Therefore, Eqs.~\eqref{Lc3} and \eqref{Lc4} hold, that is, $(A, \Delta_A)$ is a Novikov coalgebra.

This completes the proof.
\end{proof}

If the $\ZZ$-graded right Novikov algebra $(B=\oplus_{i\in\ZZ}B_i, \diamond)$ is finite-dimensional, then Eq.~(\ref{GLiebi1}) is a finite sum.

Here is an example of a \complete Lie coalgebra obtained by Theorem \ref{Constr-Lie coalgebra}.
\begin{ex}
Let $(A={\bf k}e, \Delta_A)$ be the 1-dimensional Novikov coalgebra with $\Delta_A$ defined by
\begin{eqnarray*}
\Delta_A(e)=e\otimes e.
\end{eqnarray*}
Let $(B={\bf k}[t,t^{-1}], \Delta_B)$ be the \complete right Novikov coalgebra in Example \ref{Laurent-coproduct}. Then by Theorem \ref{Constr-Lie coalgebra}, there is a \complete Lie coalgebra $(L=A\otimes B, \delta)$ with the linear map $\delta$ given by
\begin{eqnarray*}
\delta(et^j)&=&\Delta_A(e)\bullet \Delta_B(t^j)=\sum_{i\in \mathbb{Z}}(i+1)et^{-i-2}\otimes et^{j+i}-\sum_{i\in \mathbb{Z}}(i+1)et^{j+i}\otimes et^{-i-2} \\
&=&\sum_{i\in \mathbb{Z}}(j+2i+2)et^{-i-2}\otimes et^{j+i}\;\;\;\;\tforall  j\in \mathbb{Z}.
\end{eqnarray*}
\end{ex}
\subsection{Lie bialgebras from Novikov bialgebras and quadratic right Novikov algebras}
\label{ss:novbiaff}

Now we introduce the last ingredients for the construction of Lie bialgebras from Novikov bialgebras.

\begin{defi}\label{def:quad}
A bilinear form
$(\cdot,\cdot)$ on a $\ZZ$-graded vector space $B=\oplus_{i\in \ZZ}B_i$ is called {\bf graded} if there exists some $m\in \ZZ$ such that
\begin{eqnarray*}
(B_i,B_j)=0 \;\;\; \text{for all $i$, $j\in \mathbb{Z}$ satisfying $i+j+m\neq 0$.}
\end{eqnarray*}
A graded bilinear form on a $\ZZ$-graded right Novikov algebra $(B=\oplus_{i\in \ZZ}B_i, \diamond)$ is called {\bf invariant} if it satisfies
\begin{eqnarray}\label{Rbilinear1}
(a\diamond b,c)=-(a, b\diamond c+c\diamond b)\;\;\tforall
a,b,c\in B.
\end{eqnarray}
A {\bf quadratic  $\ZZ$-graded right Novikov algebra}, denoted by $(B=\oplus_{i\in \ZZ}B_i, \diamond,
(\cdot,\cdot))$, is a $\ZZ$-graded right Novikov algebra $(B,\diamond)$ together
with a symmetric invariant nondegenerate graded bilinear form
$(\cdot,\cdot)$.  In particular, when $B=B_0$, it is simply called a  {\bf quadratic right Novikov algebra}.
\end{defi}
For a quadratic $\ZZ$-graded right Novikov algebra $(B=\oplus_{i\in \ZZ}B_i, \diamond, (\cdot, \cdot))$, the nondegenerate symmetric bilinear form $(\cdot,\cdot)$ induces bilinear forms $(\cdot,\cdot)_k, k\geq 2$, defined by
{\small
\begin{equation} \label{eq:pairb}
(\cdot,\cdot)_k: (\underbrace{B\hatot \cdots \hatot
B}_{k\text{-fold}}) \ot (\underbrace{B \ot \cdots \ot
B}_{k\text{-fold}}) \to \bfk,
    \Big(\hspace{-.3cm}\sum_{i_1,\cdots,i_k,\alpha} a_{1i_1\alpha}\ot \cdots \ot a_{ki_k\alpha}, b_1\ot \cdots \ot b_k\Big)_k\coloneqq \hspace{-.4cm}\sum_{i_1,\cdots,i_k,\alpha} \prod_{\ell=1}^k(a_{\ell i_\ell\alpha}, b_\ell)
\end{equation}
}
with the notation in Eq.~\eqref{eq:ssum} and homogeneous elements $b_i\in B$.
Further the forms are {\bf left nondegenerate} in the sense that if
$$\Big(\sum_{i_1,\cdots, i_k,\alpha} a_{1i_1\alpha}\ot \cdots \ot a_{ki_k\alpha}, b_1\ot \cdots \ot b_k\Big)_k=\Big(\sum_{j_1,\cdots, j_k,\beta} b_{1j_1\beta}\ot \cdots \ot b_{kj_k\beta}, b_1\ot \cdots \ot b_k\Big)_k
$$
for all homogeneous elements $b_1, \ldots,b_k\in B$, then
$$\sum_{i_1,\cdots, i_k,\alpha} a_{1i_1\alpha}\ot \cdots \ot a_{ki_k\alpha}=\sum_{j_1,\cdots, j_k,\beta} b_{1j_1\beta}\ot \cdots \ot b_{kj_k\beta}.
$$
For brevity, we will suppress the index $k$ since the meaning will be clear from the contexts.
\begin{rmk}
\begin{enumerate}
\item Suppose that $(B=\oplus_{i\in \ZZ}B_i, \diamond,
(\cdot,\cdot))$ is a  quadratic $\ZZ$-graded right Novikov algebra. Let $\{e_p\}_{p\in \Pi}$
be a basis of $B$ consisting of homogeneous elements.
Since $(\cdot, \cdot)$ induces an isomorphism of graded vector spaces $\varphi: B=\oplus_{i\in \ZZ}B_i\rightarrow \oplus_{i\in \ZZ}B_i^\ast$, we can always find its (graded) dual basis $\{f_q\}_{q\in \Pi}$, again consisting of homogeneous elements,
associated with $(\cdot, \cdot)$, that is, $(e_p, f_q)=\delta_{p,q}$ for all $p$, $q\in \Pi$.

\item Let a $\ZZ$-graded right Novikov algebra $(B=\oplus_{i\in
\ZZ}B_i, \diamond)$ have a unit $e$. Setting $b=e$ in
Eq.~\eqref{Rbilinear1} yields $3(a,c)=0$. Thus as long as the
characteristic of ${\bf k}$ is not $3$, the bilinear form is
trivial,  showing that there does not exist any non-trivial
bilinear form $(\cdot,\cdot)$ such that $(B, \diamond, (\cdot,
\cdot))$ is a unital quadratic $\ZZ$-graded right Novikov algebra.
\end{enumerate}
\end{rmk}

\begin{ex}\label{quadratic-example}
Let $(B=B_0, \diamond)$ be the $2$-dimensional right Novikov algebra
in~\cite{BaiMeng} with a basis $\{e_1,e_2\}$ whose multiplication is given by
\begin{eqnarray*}
e_1\diamond e_1=0,~~e_1\diamond e_2=-2e_1,~~e_2\diamond e_1=e_1,~~e_2\diamond e_2=e_2.
\end{eqnarray*}
Define a bilinear form $( \cdot ,\cdot )$ on $B$ by
$(e_1,e_1)=(e_2,e_2)=0,~~(e_1,e_2)=(e_2,e_1)=1.$
This bilinear form is evidently nondegenerate symmetric and invariant, showing that $(B, \diamond, (\cdot,\cdot ))$ is a quadratic right Novikov algebra.
\end{ex}
\begin{ex}\label{Laurent-Bilinear}
Let $(B=\oplus_{i\in \mathbb{Z}}{\bf k}t^i, \diamond)$ be the $\mathbb{Z}$-graded right Novikov algebra given in Example \ref{Laurent Novikov}.
Define a bilinear form $(\cdot,\cdot)$ on $B$ by
\begin{eqnarray}\label{Laurent-Bilinear-1}
(t^i,t^j)=\delta_{i+j+1,0}\;\;\text{for all $i$, $j\in \mathbb{Z}$.}
\end{eqnarray}
It is directly checked that $(B={\bf k}[t,t^{-1}], \diamond, (\cdot, \cdot))$
is a quadratic $\mathbb{Z}$-graded right Novikov algebra.
\end{ex}
\begin{lem}\label{lem:cop}
Let $(B=\oplus_{i\in \ZZ}B_i, \diamond, (\cdot, \cdot))$ be a quadratic $\ZZ$-graded right Novikov algebra. Let $\Delta_B:B\rightarrow B\hatot B$ be the dual of $\diamond$ under the left nondegenerate bilinear form in Eq.~\eqref{eq:pairb}, so that
\begin{equation}\label{eq:coproduct-self}
(\Delta_B(a), b\otimes c)=(a, b\diamond c) ~~~~~\tforall  a,b,c\in B.
\end{equation}
Then $(B,\Delta_B)$ is a \complete right Novikov coalgebra.
\end{lem}

\begin{proof}
Note that for any $c_1\otimes c_2\otimes c_3\in B_i \otimes B_j \otimes B_k$, by the definition of $\Delta_B$, we get
\begin{eqnarray*}
\Big((\Delta_B\hatot \id)\Delta_B(a)-(\id\hatot  \widehat{\tau})(\Delta_B\hatot  \id)\Delta_B(a)-(\id\hatot  \Delta_B)\Delta_B(a)+(\id\hatot  \widehat{\tau})(\id \hatot  \Delta_B)\Delta_B(a), c_1\otimes c_2\otimes c_3\Big)\\
=(a, (c_1\diamond c_2)\diamond c_3-(c_1\diamond c_3)\diamond c_2-c_1\diamond (c_2\diamond c_3)+c_1\diamond (c_3\diamond c_2)))=0.
\end{eqnarray*}
Then the nondegeneracy of $(\cdot,\cdot)$ gives Eq. (\ref{top-Nov-Coalg-1}). Eq.~\eqref{top-Nov-Coalg-2} is similarly verified, completing the proof.
\end{proof}

\begin{rmk}
By Example \ref{Laurent-Bilinear}, the pair $(B={\bf
k}[t,t^{-1}], \diamond, (\cdot, \cdot))$ with $(\cdot, \cdot)$
given by Eq. (\ref{Laurent-Bilinear-1}) is a quadratic
$\mathbb{Z}$-graded right Novikov algebra. A direct calculation
shows that the pair $(B, \Delta_B)$ obtained from $(B={\bf
k}[t,t^{-1}], \diamond, (\cdot, \cdot))$ by Lemma \ref{lem:cop}
coincides with the \complete right Novikov coalgebra given
in Example \ref{Laurent-coproduct}.
\end{rmk}

We now give the notion and results on \complete Lie bialgebras.
\begin{defi}
A {\bf \complete Lie bialgebra}
is a triple $(L, [\cdot,\cdot],\delta)$ such that
$(L, [\cdot, \cdot])$ is a Lie algebra, $(L,
\delta)$ is a \complete Lie coalgebra, and the following
compatibility condition holds.
\begin{eqnarray} \label{lia3}
\notag \delta([a,b])=(\ad_a\hatot  \id+\id \hatot
\ad_a)\delta(b)-(\ad_b\hatot  \id+\id \hatot
\ad_b)\delta(a)\;\;\tforall  a, b\in L,
\end{eqnarray}
where $\ad_a (b)=[a,b]$ for all $a, b\in
L$.
\end{defi}
\begin{thm}\label{Gcorrespond6} Let $(A,\circ,\Delta_A)$ be a Novikov bialgebra and $(B=\oplus_{i\in \ZZ}B_i, \diamond, (\cdot, \cdot))$ be a
quadratic $\ZZ$-graded right Novikov algebra. Let
$L=A\otimes B$ be the induced Lie algebra from $(A,
\circ)$ and $(B, \diamond)$ in Theorem~\ref{pp:tensorlie}, $\Delta_B:B\rightarrow
B\hatot B$ be the linear map defined by
Eq.~\eqref{eq:coproduct-self}, and $\delta:L\rightarrow L\widehat{ \otimes }L$ be the linear map defined in Eq.~\eqref{GLiebi1}$:$
\begin{eqnarray}\label{GLiebi1b} \notag
\delta(a\otimes b)=(\id_{L\widehat{\ot} L}-\widehat{\tau})\big(\Delta_A(a)\bullet \Delta_B(b)\big) \quad \tforall  a\in A, b\in B.
\end{eqnarray}
Then $(L, [\cdot, \cdot], \delta)$ is a \complete Lie bialgebra. Further, if $(B, \diamond, (\cdot, \cdot))=({\bf k}[t,t^{-1}], \diamond, (\cdot, \cdot))$ is the quadratic $\mathbb{Z}$-graded right Novikov algebra given in Example \ref{Laurent-Bilinear}, then the converse also holds.
\end{thm}
\begin{proof}
By Lemma~\ref{lem:cop},
$(B,\Delta_B)$ is a \complete right Novikov coalgebra. Then by
Theorem \ref{Constr-Lie coalgebra}, $(L, \delta)$
is a \complete Lie coalgebra.

Let $a\otimes b, c\otimes d\in A\otimes B$ with $b\in B_i, d\in B_j$. We obtain
$$\delta([a\otimes b, c\otimes d])\\
=(\id_{L\widehat{\ot} L}-\widehat{\tau})\big(\Delta_A(a\circ c)\bullet \Delta_B(b\diamond d)\big)-(\id_{L\widehat{\ot} L}-\widehat{\tau})\big(\Delta_A(c\circ a)\bullet \Delta_B(d\diamond b)\big)
$$
and
{\small
\begin{eqnarray*}
&&(\ad_{(a\otimes b)}\hatot  \id+\id\hatot \ad_{(a\otimes b)}) \delta(c\otimes d)\\
&=&(\id_{L\widehat{\ot} L}-\widehat{\tau})\Big(\sum_{(c)}\sum_{i,j,\alpha}(a\circ c_{(1)}\otimes c_{(2)})\bullet(b \diamond d_{1i\alpha}\otimes d_{2j\alpha})-\sum_{(c)}\sum_{i,j,\alpha}(c_{(1)}\circ a\otimes c_{(2)})\bullet ( d_{1i\alpha} \diamond b\otimes d_{2j\alpha})\\
&&+\sum_{(c)}\sum_{i,j,\alpha}(c_{(1)}\otimes a\circ c_{(2)})\bullet ( d_{1i\alpha}\otimes b\diamond d_{2j\alpha})-\sum_{(c)}\sum_{i,j,\alpha}(c_{(1)}\otimes c_{(2)}\circ a)\bullet ( d_{1i\alpha}\otimes d_{2j\alpha}\diamond b)\Big).
\end{eqnarray*}
}
Let $e\in B_{k}$ and $f\in B_{l}$. Note that
{\small
\begin{eqnarray*}
&&(\widehat{\tau} \Delta_B(b\diamond d), e\otimes f)=\Big(\sum_{i,j,\alpha} (b\diamond d)_{2j\alpha}\hatot  (b\diamond d)_{1i\alpha}, e\otimes f\Big)=(b\diamond d, f\diamond e)=(d, b\diamond (f\diamond e)),\\
&&\Big(\sum_{i,j,\alpha} (b\diamond d_{1i\alpha}\otimes d_{2j\alpha}), e\otimes f\Big)=\sum_{i,j,\alpha} (d_{1i\alpha}\otimes d_{2j\alpha}, ( b\diamond e)\otimes f)=(d, (b\diamond e)\diamond f)\\
&&=\Big(d, (b\diamond f)\diamond e+b\diamond (e\diamond f)-b\diamond (f\diamond e)\Big).
\end{eqnarray*}
}
By the nondegeneracy of $(\cdot,\cdot)$, we have
\begin{equation}
\label{bialg-eq1}\sum_{i,j,\alpha} (b\diamond d_{1i\alpha}\otimes d_{2j\alpha})=-\widehat{\tau} \Delta_B(b\diamond d)+\sum_{i,j,\alpha} g_{i\alpha}\otimes h_{j\alpha},
\end{equation}
where $\sum_{i,j,\alpha} g_{i\alpha}\otimes h_{j\alpha}\in B\hatot  B$ is chosen so that $(\sum_{i,j,\alpha} g_{i\alpha}\otimes h_{j\alpha}, e\otimes f)=(d, (b\diamond f)\diamond e+b\diamond (e\diamond f))$.
Similarly, we obtain
{\small
\begin{eqnarray}
\label{bialg-eq2}&&\sum_{i,j,\alpha} (d_{1i\alpha}\otimes b\diamond d_{2j\alpha})=\Delta_B(b\diamond d),\;\;\sum_{i,j,\alpha} (d_{1i\alpha}\otimes d_{2j\alpha}\diamond b)=-\Delta_B(b\diamond d)-\sum_{i,j,\alpha} k_{i\alpha}\otimes l_{j\alpha},\\
\label{bialg-eq3}&&\sum_{i,j,\alpha} (d_{1i\alpha}\diamond b\otimes d_{2j\alpha} )
=\widehat{\tau}\Delta_B(b\diamond d)-\sum_{i,j, \alpha}g_{i\alpha}\otimes h_{j\alpha}+\Delta_B(b\diamond d)+\sum_{i,j,\alpha} k_{i\alpha}\otimes l_{j\alpha},\\
\label{bialg-eq5}&&\sum_{i,j,\alpha} (b_{1i\alpha}\otimes d \diamond b_{2j\alpha})=\Delta_B(d\diamond b),\;\;\sum_{i,j,\alpha} (d \diamond b_{1i\alpha}\otimes b_{2j\alpha})=-\widehat{\tau} \Delta_B(d\diamond b)+\sum_{i,j,\alpha} g_{i\alpha}\otimes h_{j\alpha},\\
\label{bialg-eq7}&&\sum_{i,j,\alpha} (b_{1i\alpha}\diamond d\otimes b_{2j\alpha})
=\widehat{\tau} \Delta_B(d\diamond b)-\sum_{i,j,\alpha}g_{i\alpha}\otimes h_{j\alpha}+\Delta_B(d\diamond b)+\sum_{i,j,\alpha} k_{i\alpha}\otimes l_{j\alpha},\\
\label{bialg-eq8}&&\sum_{i,j,\alpha} (b_{1i\alpha}\otimes b_{2j\alpha}\diamond d)=-\Delta_B(d\diamond b)-\sum_{i,j,\alpha} k_{i\alpha}\otimes l_{j\alpha},
\end{eqnarray}}
where $\sum_{i,j,\alpha} k_{i\alpha}\otimes l_{j\alpha}\in B\hatot  B$ is chosen so that $(\sum_{i,j,\alpha} k_{i\alpha}\otimes l_{j\alpha}, e\otimes f)=(d, e\diamond (f\diamond b))$.

Applying Eqs. (\ref{bialg-eq1})-(\ref{bialg-eq8}) and Eqs. (\ref{Lb5})-(\ref{Lb7}), we obtain
\small{
\begin{eqnarray*}
&&\delta([a\otimes b, c\otimes d])-(\ad_{(a\otimes b)}\hatot  \id+\id\hatot \ad_{(a\otimes b)})\delta(c\otimes d)+(\ad_{(c\otimes d)}\hatot  \id+\id\hatot \ad_{(c\otimes d)})\delta(a\otimes b)\\
&=&(\id_{L\widehat{\ot} L}-\widehat{\tau})((\Delta_A(a\circ c)-\sum_{(c)}c_{(1)}\otimes (a\star c_{(2)})- \sum_{(c)} c_{(2)}\otimes (a\star c_{(1)})-\sum_{(a)} a_{(1)}\circ c\otimes a_{(2)})\bullet \Delta_B(b\diamond d))\\
&&-(\id_{L\widehat{\ot} L}-\widehat{\tau})((\Delta_A(c\circ a)-\sum_{(c)} c_{(1)}\circ a\otimes c_{(2)}- \sum_{(a)} a_{(1)}\otimes (c\star a_{(2)})-\sum_{(a)} a_{(2)}\otimes c\star a_{(2)})\bullet \Delta_B(d\diamond b))\\
&&-\sum_{(c)}(a\star c_{(1)}\otimes c_{(2)}-c_{(2)}\otimes a\star c_{(1)}-c\star a_{(1)}\otimes a_{(2)}+a_{(2)}\otimes c\star a_{(1)})\bullet \sum_{i,j,\alpha}g_{i\alpha}\otimes h_{j\alpha}\\
&&+\sum_{(c)}(c_{(1)}\circ a\otimes c_{(2)}-c_{(1)}\otimes c_{(2)}\circ a+c_{(2)}\circ a\otimes c_{(1)}-c_{(2)}\otimes c_{(1)}\circ a\\
&&-a_{(1)}\circ c\otimes  a_{(2)}+a_{(1)}\otimes a_{(2)}\circ c-a_{(2)}\circ c\otimes a_{(1)}+a_{(2)}\otimes a_{(1)}\circ c)\bullet \sum_{i,j,\alpha}k_{i\alpha}\otimes l_{j\alpha}
= 0.
\end{eqnarray*}
}
Therefore, $(L, [\cdot, \cdot], \delta)$ is a \complete Lie bialgebra.

If $(B, \diamond, (\cdot, \cdot))=({\bf k}[t,t^{-1}], \diamond, (\cdot, \cdot))$
 and $(L, [\cdot, \cdot], \delta)$ is a \complete Lie bialgebra, then
 $(A, \circ)$ is a Novikov algebra and $(A, \Delta_A)$ is a Novikov coalgebra by Theorems \ref{correspond1} and \ref{Constr-Lie coalgebra} respectively. Note that in this case, $\delta$ is given by Eq. (\ref{Lie-coproduct}).
 Then we only need to verify Eqs. (\ref{Lb5})-(\ref{Lb7}).

We compute
\begin{eqnarray*}
0&=&\delta([at^j, bt^k])-(\ad_{at^j}\widehat{\otimes}\id+\id\widehat{\otimes} \ad_{at^j}) \delta(bt^k)+(\ad_{bt^k}\widehat{\otimes}\id+\id\widehat{\otimes} \ad_{bt^k}) \delta(at^j)\\
&=&\sum_{i\in \mathbb{Z}}\sum_{(a\circ b)}\Big(j(i+1)({(a\circ b)_{(1)}}t^{-i-2}\otimes {(a\circ b)_{(2)}}t^{j+k+i-1}-{(a\circ b)_{(2)}}t^{j+k+i-1}\otimes{(a\circ b)_{(1)}}t^{-i-2})\\
&&-k(i+1)({(b\circ a)_{(1)}}t^{-i-2}\otimes {(b\circ a)_{(2)}}t^{j+k+i-1}-{(b\circ a)_{(2)}}t^{j+k+i-1}\otimes {(b\circ a)_{(1)}}t^{-i-2})\Big)\\
&&- \sum_{i\in \mathbb{Z}}\sum_{(b)}\Big((i+1)(j(a\circ b_{(1)})t^{j-i-3}+(i+2)(b_{(1)}\circ a)t^{j-i-3})\otimes {b_{(2)}}t^{k+i}\\
&&+(i+1)b_{(1)}t^{-i-2}\otimes (j(a\circ b_{(2)})t^{j+k+i-1}-(k+i)(b_{(2)}\circ a)t^{j+k+i-1})\\
&&-(i+1)(j(a\circ b_{(2)})t^{j+k+i-1}-(k+i)(b_{(2)}\circ a)t^{j+k+i-1})
\otimes b_{(1)}t^{-i-2}\\
&&-(i+1)b_{(2)}t^{k+i}\otimes (j(a\circ b_{(1)})t^{j-i-3}+(i+2)(b_{(1)}\circ a)t^{j-i-3})\Big)\\
&&+\sum_{i\in \mathbb{Z}}\sum_{(a)}\Big((i+1)(k(b\circ a_{(1)})t^{k-i-3}+(i+2)(a_{(1)}\circ b)t^{k-i-3})\otimes a_{(2)}t^{j+i}\\
&&+(i+1)a_{(1)}t^{-i-2}\otimes (k(b\circ a_{(2)})t^{j+k+i-1}-(j+i)(a_{(2)}\circ b)t^{j+k+i-1})\\
&&-(i+1)(k(b\circ a_{(2)})t^{j+k+i-1}-(j+i)(a_{(2)}\circ b)t^{j+k+i-1})\otimes a_{(1)}t^{-i-2}\\
&&-(i+1)a_{(2)}t^{j+i}\otimes (k(b\circ a_{(1)})t^{k-i-3}+(i+2)(a_{(1)}\circ b)t^{k-i-3})\Big).
\end{eqnarray*}
Let $k=0$ and $j=2$. Comparing the coefficients of $t^{-1}\otimes 1$, we obtain
\begin{eqnarray}\label{Nov-bialgebra-1} \notag
\tau \Delta_A(a\circ b)=(L_{A,\star}(a)\otimes \id)(\Delta_A(b)+\tau\Delta_A(b))+(\id\otimes R_A(b))\tau\Delta_A(a) \;\;\;\tforall  a, b \in A.
\end{eqnarray}
Then we obtain Eq.~\eqref{Lb5}.
Similarly, setting $j=k=0$ and comparing the coefficients of $t^{-3}\otimes 1$ yield Eq. (\ref{Lb7}); setting $j=k=1$ and comparing the coefficients of $t^{-1}\otimes 1$ by Eq.~\eqref{Lb5} give Eq. (\ref{Lb6}).

This completes the proof.
\end{proof}
\begin{rmk}\label{rk:char}
\begin{enumerate}
\item
Theorem~\ref{Gcorrespond6} shows that the notion of the Novikov bialgebra can be characterized by the condition that its ``affinization" with the quadratic $\ZZ$-graded right Novikov algebra $(\bfk[t,t^{-1}],\diamond, (\cdot,\cdot))$ is a Lie bialgebra. As in the case of Novikov algebras (Theorem~\ref{correspond1}), this characterization on the one hand gives a general construction of infinite-dimensional Lie bialgebras. On the other hand, it shows the significance of the Novikov bialgebra.
\item In view of the operadic interpretation of Theorem~\ref{correspond1}, that the left and right Novikov algebras are the Koszul dual of each other~\cite{GK,LV}, Theorem~\ref{Gcorrespond6} suggests that the Novikov bialgebra and quadratic right Novikov algebra (instead of the right Novikov bialgebra) are the operadic Koszul dual of each other as dioperads or properads, even though the general results in~\cite{Ga,Va} do not apply to this situation. Adapting Theorem~\ref{Gcorrespond6} to such results should provide a general procedure to construct Lie bialgebras.
\end{enumerate}
\end{rmk}

In Section~\ref{sec:nybe}, we will give a general method to apply Theorem~\ref{Gcorrespond6} in constructing Lie bialgebras. For now, we present a simple example.

\begin{ex}
Let $(A, \circ, \Delta_A)$ be the Novikov bialgebra given in
Example \ref{ex-Nov-bi} and $(B, \diamond, (\cdot, \cdot))=({\bf
k}[t,t^{-1}], \diamond, (\cdot, \cdot))$ be the quadratic
$\mathbb{Z}$-graded right Novikov algebra given in Example
\ref{Laurent-Bilinear}. Then by Theorem \ref{Gcorrespond6},
there is a \complete Lie bialgebra $(L=A\otimes B, [\cdot, \cdot], \delta)$ given by
\begin{eqnarray*}
&&[e_1t^i, e_1t^j]=(i-j)e_1t^{i+j-1},\;\;[e_1t^i, e_2t^j]=-je_2t^{i+j-1},\;\;[e_2t^i,e_2t^j]=0 \;\; \tforall  i, j\in \ZZ,\\
&&\delta(e_1t^j)=\sum_{i\in \ZZ}\lambda (j+2i+2)e_2t^{-i-2}\otimes
e_2t^{j+i},\;\; \delta(e_2t^j)=0 \;\;\tforall j\in
\ZZ.
\end{eqnarray*}
\end{ex}
\section{Characterizations of Novikov bialgebras and Novikov Yang-Baxter equation}
\label{sec:nybe}

In Section~\ref{ss:mpair}, we use matched pairs and Manin triples
of Novikov algebras to give two equivalent conditions of Novikov
bialgebras.  We then give in Section~\ref{ss:manin} a natural
construction of a Manin triple of Lie algebras from the coupling
of a Manin triple of Novikov algebras with a quadratic right
Novikov algebra. Section~\ref{ss:nybe} focuses on a special class
of Novikov bialgebras in analogous to the coboundary Lie
bialgebras, leading to the introduction of the Novikov
Yang-Baxter equation (\nybe). A skewsymmetric solution of the
\nybe  in a Novikov algebra gives a Novikov bialgebra. The notions of an $\mathcal O$-operator of a Novikov algebra, a
\ndend algebra and a quasi-Frobenius Novikov algebra are introduced to interpret and construct solutions of the \nybe.

In this section, we assume that all algebras, representations and vector spaces are of finite dimension, although some results still hold in the infinite-dimensional case.
\subsection{Characterizations of Novikov bialgebras}
\label{ss:mpair}
We first give some background on representations of Novikov algebras needed for matched pairs of Novikov algebras.
\vspb
\begin{defi}\cite{O3}
A {\bf representation} of a Novikov algebra $(A,\circ)$ is a triple $(V, l_A,r_A)$, where $V$ is a vector space and   $l_A$, $r_A: A\rightarrow {\rm
End}_{\bf k}(V)$ are linear maps satisfying
\vspb
\begin{eqnarray}
\label{lef-mod1}\notag
&l_A(a\circ b-b\circ a)v=l_A(a)l_A(b)v-l_A(b)l_A(a)v,&\\
\label{lef-mod2}\notag
&l_A(a)r_A(b)v-r_A(b)l_A(a)v=r_A(a\circ b)v-r_A(b)r_A(a)v,&\\
\label{Nov-mod1}\notag
&l_A(a\circ b)v=r_A(b)l_A(a)v,
&\\
\label{Nov-mod2}\notag
&r_A(a)r_A(b)v=r_A(b)r_A(a)v &
\hspace{-2.7cm} \quad \tforall  a, b\in A, v\in V.
\vspc
\end{eqnarray}
\end{defi}
\vspb
Note that $(A, L_A, R_A)$ is a representation of $(A,\circ)$, called the \textbf{adjoint representation} of $(A,\circ)$.

\vspa
\begin{pro}\label{pro:semi}
Let $(A,\circ)$ be a Novikov algebra. Let $V$ be a vector space
and $l_A,r_A: A\to {\rm End}_{\bf k}(V)$ be   linear maps.
Define a binary operation $\bullet$ on $A\oplus V$ by
\vspa
    $$(a+u)\bullet(b+v)\coloneqq a\circ b+l_A(a)v+r_A(b)u \quad\tforall  a,b\in A,u,v\in V.$$
Then $(V,l_A,r_A)$ is a representation of $(A,\circ)$ if and only
if $(A\oplus V,\bullet)$ is a Novikov algebra, called the
{\bf semi-direct product} of $A$ by $V$ and denoted by
$A\ltimes_{l_A,r_A} V$ or simply $A\ltimes V$.
\end{pro}

This result will follow as a special case of  Proposition~\ref{Matched pair} for matched pairs of Novikov algebras, when $V$ is regarded as a Novikov algebra with the zero multiplication.

Let $(A,\circ)$ be a Novikov algebra and $V$ be a vector space.
For a linear map $\varphi:A\rightarrow\mathrm{End}_{
\bf k}(V)$, define a linear map
$\varphi^{*}:A\rightarrow\mathrm{End}_{\bf k}(V^{*})$ by
\vspa
$$\langle\varphi^{*}(a)f,v\rangle=-\langle f,\varphi(a)v\rangle \tforall a\in A, f\in V^{*},v\in V,$$
where $\langle\cdot, \cdot\rangle$ is the usual pairing between $V$ and $V^*$.

\begin{pro} \label{pp:dualrep}
Let $(A,\circ)$ be a Novikov algebra and $(V, l_A, r_A)$ be a
representation of $(A,\circ)$. Then $(V^\ast, l_A^\ast+r_A^\ast, -r_A^\ast)$ is
a representation of $(A,\circ)$.
\end{pro}
\begin{proof}
It is a straightforward check.
\end{proof}

\begin{ex} \label{ex:dualrep}
The adjoint representation of a Novikov algebra $(A,\circ)$ gives the representation
$(A^\ast, L_A^\ast+R_A^\ast, -R_A^\ast)$.
\end{ex}

\begin{rmk} \label{rk:kuper}
Let $(A,\cdot)$ be an
algebraic structure with linear maps $L_A$ and $R_A$ which give a natural representation of $A$ on itself. As the central notion introduced in the comprehensive work of Kupershmidt~\cite{Ku} on algebraic structures, an algebraic structure is called {\bf
proper} if, for every such algebra $(A,\cdot)$, there is a representation on the linear dual $A^*$ that can be given by a nonzero linear combination of $L_A^*$ and $R_A^*$.
This property is essential for example for the algebraic structure to have a bialgebra theory
comparable to the Lie bialgebra and infinitesimal associative bialgebra.
It was incorrectly perceived in~\cite{Ku}\footnote{See \S 19 where he wrote ``The
    moral is that Novikov algebras are not proper".}
that the Novikov algebra is not proper, which would have been an obstacle for a bialgebra theory for the Novikov algebra. Proposition~\ref{pp:dualrep} rectifies this perception. In other words, the Novikov algebra is in fact proper. Thus it is possible to establish a reasonable bialgebra theory, as presented in this paper.
\end{rmk}

We recall matched pairs of Novikov algebras.

\begin{pro}\label{Matched pair} \cite{Hong}
Let $(A,\circ)$ and $(B,\cdot)$ be Novikov algebras. Suppose that $(B,l_A, r_A)$ is a representation of $(A,\circ)$,
$(A,l_B, r_B)$ is a representation of $(B,\cdot)$ and the following conditions are satisfied.
\begin{eqnarray}
&&\label{matchpair1}l_B(x)(a\circ b)=-l_B(l_A(a)x-r_A(a)x)b+(l_B(x)a-r_B(x)a)\circ b+r_B(r_A(b)x)a+a\circ(l_B(x)b),\\
&&\label{matchpair2}r_B(x)(a\circ b-b\circ a)=r_B(l_A(b)x)a-r_B(l_A(a)x)b+a\circ (r_B(x)b)-b\circ (r_B(x)a),\\
&&\label{matchpair3}l_A(a)(x\cdot y)=-l_A(l_B(x)a-r_B(x)a)y+(l_A(a)x-r_A(a)x)\cdot y+r_A(r_B(y)a)x+x\cdot(l_A(a)y),\\
&&\label{matchpair4}r_A(a)(x\cdot y-y\cdot x)=r_A(l_B(y)a)x-r_A(l_B(x)a)y+x\cdot(r_A(a)y)-y\cdot(r_A(a)x),\\
&&\label{matchpair5}(l_B(x)a)\circ b+l_B(r_A(a)x)b=(l_B(x)b)\circ a+l_B(r_A(b)x)a,\\
&&\label{matchpair6}(r_B(x)a)\circ b+l_B(l_A(a)x)b=r_B(x)(a\circ b),\\
&&\label{matchpair7}l_A(r_B(x)a)y+(l_A(a)x)\cdot y=l_A(r_B(y)a)x+(l_A(a)y)\cdot x,\\
~~&&\label{matchpair8}l_A(l_B(x)a)y+(r_A(a)x)\cdot y=r_A(a)(x\cdot y) \tforall a, b\in A, x, y\in B.
\end{eqnarray}
Then there is a Novikov algebra structure on the direct sum $A\oplus B$ of the underlying vector spaces of $A$ and $B$ given by
\vspb
\begin{eqnarray}
(a+x)\bullet (b+y):=\big(a\circ b+l_B(x)b+r_B(y)a\big)+\big(x\cdot y+l_A(a)y+r_A(b)x\big)
\vspa
\end{eqnarray}
for all $a,b\in A, x,y\in B$. We call the resulting sextuple $(A,
B, l_A, r_A, l_B, r_B)$ a {\bf matched pair of Novikov algebras}.
Conversely, any Novikov algebra that can be decomposed into a linear direct sum of two Novikov subalgebras is obtained from a matched pair of Novikov algebras.
\end{pro}

Similar to Definition~\ref{def:quad}, we give the following notion.

\begin{defi}\label{Novbilinear}
Let $(A,\circ)$ be a Novikov algebra. A bilinear form
$\mathcal{B}(\cdot,\cdot)$  on $A$ is called {\bf invariant} if it satisfies
\vspc
\begin{eqnarray}\label{bilinear1}
\mathcal{B}(a\circ b,c)=-\mathcal{B}(b, a\star c)\;\;\tforall
a,b,c\in A.
\vspa
\end{eqnarray}
A {\bf quadratic Novikov algebra}, denoted by $(A,
\circ,\mathcal{B}(\cdot,\cdot))$, is a Novikov algebra $(A,\circ)$
together with a nondegenerate symmetric invariant bilinear form
$\mathcal{B}(\cdot,\cdot)$.
\end{defi}
\vspb
\begin{rmk}\label{quadratic-opp}
\begin{enumerate}
\item It is easy to see that if $(B, \diamond, (\cdot, \cdot))$ is a finite-dimensional quadratic right Novikov algebra, then with the multiplication
 $a\circ  b:=b\diamond a$ for all $a$, $b\in B$, the triple
$(B, \circ, (\cdot, \cdot))$ is a quadratic Novikov algebra. \item
The invariant condition of the bilinear form on a Novikov algebra
defined by Eq.~\eqref{bilinear1} is different from those studied in~\cite{BM4, G, Le, Z, ZC}.
\end{enumerate}
\end{rmk}

Motivated by the Manin triples of Lie algebras \cite{CP},
 we give the following notion.
\begin{defi}
 A {\bf (standard) Manin triple of Novikov algebras} is a triple of Novikov algebras $(A=A_1\oplus A_1^*,(A_1,\circ),(A_1^\ast,\cdot))$ for which
\begin{enumerate}
    \item as a vector space, $A$ is the direct sum of $A_1$ and $A_1^\ast$;
    \item $(A_1,\circ)$ and $(A_1^\ast,\cdot)$ are Novikov subalgebras of $A$;
    \item the bilinear form on $A=A_1\oplus A_1^\ast$ defined by
\begin{eqnarray}\label{bilinear}
\mathcal{B}(a+f,b+g):=\langle f, b\rangle+\langle g, a\rangle
\;\;\tforall ~~~~~~~~~~a,~~b\in A_1,~~f,~~g\in A_1^\ast,
\end{eqnarray}
is invariant.
\end{enumerate}
\end{defi}

Obviously, the bilinear form $\mathcal{B}(\cdot,\cdot)$ defined by Eq.~(\ref{bilinear}) is symmetric and nondegenerate. Thus $(A, \circ, \mathcal{B}(\cdot,\cdot))$ is indeed a quadratic Novikov algebra. Manin triples of Novikov algebras is equivalent to certain matched pairs of Novikov algebras.
\vspa
\begin{thm}\label{thmm1}
Let $(A,\circ)$ be a Novikov algebra. Suppose that there is also a
Novikov algebra structure $\cdot$ on its linear dual space $A^\ast$. Then
there is a Manin triple $(A\oplus A^*,(A,\circ),(A^\ast,\cdot))$
of Novikov algebras if and only if
$(A, A^\ast, L_{A,\star}^\ast, -R_A^\ast, L_{A^\ast,\star}^\ast,
-R_{A^\ast}^\ast)$ is a matched pair of Novikov algebras, where
$L_{A,\star}^\ast\coloneqq L_A^\ast+R_A^\ast$ and
$L_{A^\ast,\star}^\ast\coloneqq L_{A^\ast}^\ast+R_{A^\ast}^\ast$.
\end{thm}
\vspc
\begin{proof}
The proof is direct, following that of \cite[Theorem 2.2.1]{Bai2} for associative algebras.
\end{proof}
\vspc
\begin{thm}
Let $(A,\circ)$ be a Novikov algebra and let $\Delta:A\to A\ot A$ be a linear map. Suppose that the dual of $\Delta$ also gives a
Novikov algebra structure $\cdot$ on $A^\ast$. Then $(A, A^\ast,
L_{A,\star}^\ast, -R_A^\ast, L_{A^\ast,\star}^\ast,$ $
-R_{A^\ast}^\ast)$ is a matched pair of Novikov algebras if and
only if $(A, \circ, \Delta)$ is a Novikov bialgebra.
\label{thm1}
\end{thm}
\vspe
\begin{proof}
Since the dual of $\Delta$ gives a Novikov algebra structure
$\cdot$ on $A^\ast$, $(A, \Delta)$ is a Novikov coalgebra. Let
$\{e_1, \ldots, e_n\}$ be a basis of $A$ and $\{e_1^\ast, \ldots,
e_n^\ast\}$ be its dual basis. Denote
\vspb
$$e_\alpha\circ e_\beta=\sum_{\gamma=1}^nc_{\alpha\beta}^\gamma e_\gamma,\;\;\Delta(e_\gamma)=\sum_{\alpha,\beta=1}^nd_{\alpha\beta}^\gamma e_\alpha\otimes e_\beta.
\vspc
$$
Then $e_\alpha^\ast \cdot e_\beta^\ast=\sum_{\gamma=1}^n d_{\alpha\beta}^\gamma e_\gamma^\ast$ and
\vspb
\small{
$$R_A^\ast(e_\alpha)e_\beta^\ast=-\sum_{\gamma=1}^nc_{\gamma \alpha}^\beta e_\gamma^\ast,~~L_A^\ast(e_\alpha)e_\beta^\ast=-\sum_{\gamma=1}^nc_{\alpha \gamma}^\beta e_\gamma^\ast,{R_{A^\ast}^\ast}(e_\alpha^\ast)e_\beta=-\sum_{\gamma=1}^nd_{\gamma \alpha}^\beta e_\gamma,~~{L_{A^\ast}^\ast}(e_\alpha^\ast)e_\beta=-\sum_{\gamma=1}^nd_{\alpha\gamma}^\beta e_\gamma.
\vspb
$$}
Hence $L_{A, \star}^\ast(e_\alpha)e_\beta^\ast=-\sum_{\gamma=1}^n(c_{\gamma \alpha}^\beta+c_{\alpha \gamma}^\beta)e_\gamma^\ast$ and $L_{A^\ast,\star}^\ast(e_\alpha^\ast)e_\beta=-\sum_{\gamma=1}^n(d_{\gamma \alpha}^\beta+d_{\alpha \gamma}^\beta)e_\gamma$.

Take $l_A=L_{A,\star}^\ast$, $r_A=-R_A^\ast$, $l_B=L_{A^\ast, \star}^\ast$, $r_B=-R_{A^\ast}^\ast$ in Eqs.~(\ref{matchpair1})-(\ref{matchpair8}). Setting $a=e_\alpha$, $b=e_\beta$, $x=e_\gamma^\ast$ in Eqs.~(\ref{matchpair1}), (\ref{matchpair2}), (\ref{matchpair5}) and (\ref{matchpair6}), and comparing the coefficients of $e_\ell$ in these equalities, we obtain
\vspb
\begin{eqnarray}
&&\label{corresp1}\sum_{\nu=1}^n c_{\alpha\beta}^\nu(d_{\ell \gamma}^\nu+d_{\gamma \ell}^\nu)=\sum_{\nu=1}^n\big((2c_{\nu\alpha}^\gamma+c_{\alpha \nu}^\gamma)(d_{\ell\nu}^\beta+d_{\nu \ell}^\beta)+(2d_{\nu\gamma}^\alpha+d_{\gamma \nu}^\alpha)c_{\nu\beta}^\ell-c_{\nu\beta}^\gamma d_{\ell\nu}^\alpha+(d_{\nu\gamma}^\beta+d_{\gamma \nu}^\beta)c_{\alpha \nu}^\ell\big),\\
&&\label{corresp2}\sum_{\nu=1}^n(c_{\alpha \beta}^\nu-c_{\beta\alpha}^\nu)d_{\ell \gamma}^\nu=\sum_{\nu=1}^n\big((c_{\nu\alpha}^\gamma+c_{\alpha \nu}^\gamma)d_{\ell \nu}^\beta-(c_{\nu\beta}^\gamma+c_{\beta \nu}^\gamma)d_{\ell \nu}^\alpha +d_{\nu \gamma}^\beta c_{\alpha \nu}^\ell-d_{\nu \gamma}^\alpha c_{\beta \nu}^\ell\big),\\
&&\label{corresp3}\sum_{\nu=1}^n\big((d_{\nu \gamma}^\alpha+d_{\gamma\nu}^\alpha)c_{\nu \beta}^\ell+(d_{\ell\nu}^\beta+d_{\nu \ell}^\beta)c_{\nu \alpha}^\gamma\big)=\sum_{\nu=1}^n\big((d_{\nu \gamma}^\beta+d_{\gamma\nu}^\beta)c_{\nu \alpha}^\ell+(d_{\ell\nu}^\alpha+d_{\nu \ell}^\alpha)c_{\nu \beta}^\gamma\big),\\
&&\label{corresp4}\sum_{\nu=1}^n(d_{\nu \gamma}^\alpha c_{\nu \beta}^\ell+(c_{\nu \alpha}^\gamma+c_{\alpha \nu}^\gamma)(d_{\nu \ell}^\beta+d_{\ell \nu}^\beta))=\sum_{\nu=1}^nc_{\alpha\beta}^\nu d_{\ell \gamma}^\nu.
\vspb
\end{eqnarray}
Setting $x=e_\alpha^\ast$, $y=e_\beta^\ast$ and $a=e_\gamma$ in Eqs.~(\ref{matchpair3}), (\ref{matchpair4}), (\ref{matchpair7}) and (\ref{matchpair8}), and comparing the coefficients of $e_\ell^\ast$, we derive
\vspb
\begin{eqnarray}
&&\label{corresp5}\sum_{\nu=1}^n d_{\alpha\beta}^\nu(c_{\ell \gamma}^\nu+c_{\gamma \ell}^\nu)=\sum_{\nu=1}^n\big((2d_{\nu \alpha}^\gamma+d_{\alpha \nu}^\gamma)(c_{\ell \nu}^\beta+c_{\nu \ell}^\beta)+(2c_{\nu \beta}^\alpha+c_{\gamma \nu}^\alpha)d_{\nu \beta}^\ell-d_{\nu \beta}^\gamma c_{\ell\nu}^\alpha+(c_{\nu\gamma}^\beta+c_{\gamma \nu}^\beta)d_{\alpha \nu}^\ell\big),\\
&&\label{corresp6}\sum_{\nu=1}^n(d_{\alpha \beta}^\nu-d_{\beta\alpha}^\nu)c_{\ell\gamma}^\nu=\sum_{\nu=1}^n\big((d_{\nu\alpha}^\gamma+d_{\alpha \nu}^\gamma)c_{\ell\nu}^\beta-(d_{\nu\beta}^\gamma+d_{\beta \nu}^\gamma)c_{\ell\nu}^\alpha+c_{\nu\gamma}^\beta d_{\alpha \nu}^\ell-c_{\nu\gamma}^\alpha d_{\beta \nu}^\ell\big),\\
&&\label{corresp7}\sum_{\nu=1}^n((c_{\nu \gamma}^\alpha+c_{\gamma \nu}^\alpha)d_{\nu \beta}^\ell+(c_{\ell\nu}^\beta+c_{\nu \ell}^\beta)d_{\nu \alpha}^\gamma)=\sum_{\nu=1}^n\big((c_{\nu \gamma}^\beta+c_{\gamma\nu}^\beta)d_{\nu \alpha}^\ell+(c_{\ell\nu}^\alpha+c_{\nu \ell}^\alpha)d_{\nu \beta}^\gamma\big),\\
&&\label{corresp8}\sum_{\nu=1}^n\big(c_{\nu \gamma}^\alpha d_{\nu \beta}^\ell+(d_{\nu \alpha}^\gamma+d_{\alpha\nu}^\gamma)(c_{\nu \ell}^\beta+c_{\ell\nu}^\beta)\big)=\sum_{\nu=1}^nd_{\alpha \beta}^\nu c_{\ell\gamma}^\nu.
\end{eqnarray}

Further, setting $a=e_\alpha$ and $b=e_\beta$ in Eqs.~(\ref{Lb5})-(\ref{Lb7}), and comparing the coefficients of $e_\ell\otimes e_\gamma$, we have
\vspb
\begin{eqnarray}
&&\label{corresp10}\sum_{\nu=1}^nc_{\beta \alpha}^\nu d_{\ell \gamma}^\nu=\sum_{\nu=1}^n\big(d_{\nu \gamma}^\beta c_{\nu \alpha}^\ell+(d_{\ell\nu}^\alpha+d_{\nu \ell}^\alpha)(c_{\beta\nu}^\gamma+c_{\nu \beta}^\gamma)\big),\\
&&\label{corresp11}\sum_{\nu=1}^n(d_{\nu \ell}^\beta(c_{\alpha\nu}^\gamma+c_{\nu \alpha}^\gamma)-d_{\nu \gamma}^\beta(c_{\alpha\nu}^\ell+c_{\nu \alpha}^\ell))=\sum_{\nu=1}^n\big(d_{\nu \ell}^\alpha(c_{\beta\nu }^\gamma+c_{\nu \beta}^\gamma)-d_{\nu \gamma}^\alpha(c_{\beta\nu}^\ell+c_{\nu \beta}^\ell)\big),\\
&&\label{corresp12}\sum_{\nu=1}^n((d_{\nu \ell}^\beta+d_{\ell \nu}^\beta)c_{\nu \alpha}^\gamma-(d_{\gamma\nu}^\beta+d_{\nu \gamma}^\beta)c_{\nu \alpha}^\ell)=\sum_{\nu=1}^n\big((d_{\nu\ell}^\alpha+d_{\ell \nu}^\alpha)c_{\nu \beta}^\gamma-(d_{\gamma\nu}^\alpha+d_{\nu \gamma}^\alpha)c_{\nu \beta}^\ell)\big).
\vspb
\end{eqnarray}

Now setting $\alpha=\beta$, $\beta=\alpha$ in Eq.~(\ref{corresp10}), we find that Eq.~(\ref{corresp10}) is equivalent to Eq.~(\ref{corresp4}). Moreover, by setting $\beta=\ell$, $\alpha=\gamma$, $\ell=\alpha$ and $\gamma=\beta$ in Eq.~(\ref{corresp10}), we obtain that Eq.~(\ref{corresp10}) is equivalent to Eq.~(\ref{corresp8}). Similarly, Eq.~(\ref{corresp11}) (resp. Eq.~(\ref{corresp12})) is equivalent to Eq.~(\ref{corresp7}) (resp. Eq.~(\ref{corresp3}).

By Eq.~(\ref{corresp10}), we obtain
\vspb
\begin{eqnarray}\label{corresp13}
{\small \ \ \ \ \ \sum_{\nu=1}^nc_{\alpha \beta}^\nu(d_{\ell\gamma}^\nu+d_{\gamma \ell}^\nu)=\sum_{\nu=1}^n(d_{\nu \gamma }^\alpha c_{\nu \beta}^\ell+(d_{\ell\nu}^\beta+d_{\nu \ell}^\beta)(c_{\alpha\nu}^\gamma +c_{\nu \alpha}^\gamma ))+\sum_{\nu=1}^n(d_{\nu \ell}^\alpha c_{\nu \beta}^\gamma +(d_{\gamma \nu}^\beta+d_{\nu \gamma }^\beta)(c_{\alpha\nu}^\ell+c_{\nu \alpha}^\ell)).}
\end{eqnarray}
Adding the two sides of Eqs.~(\ref{corresp13}) and
~(\ref{corresp12}) and rearranging the terms, we obtain
Eq.~(\ref{corresp1}). Therefore Eqs.~(\ref{corresp10}) and (\ref{corresp12}) imply (\ref{corresp1}). Similarly,
Eqs.~(\ref{corresp10}) and (\ref{corresp11}) imply Eq.~(\ref{corresp2});
Eqs.~ (\ref{corresp10}) and
(\ref{corresp11}) imply Eq.~(\ref{corresp5}); and Eqs.~(\ref{corresp10}) and (\ref{corresp12}) imply Eq.~(\ref{corresp6}). This completes the proof.
\end{proof}
\vspb
\begin{rmk}
By the proof of Theorem \ref{thm1}, we see that for the matched pair
$(A, A^\ast, L_{A,\star}^\ast, -R_A^\ast,$ $L_{A^\ast, \star}^\ast,
-R_{A^\ast}^\ast)$ of Novikov algebras, the following implications hold.
\vspb
\begin{eqnarray*}
&&{\rm Eq.~(\ref{matchpair6})}\Longleftrightarrow {\rm Eq.~(\ref{matchpair8})},~~~~\\
&&{\rm  Eq.~(\ref{matchpair5})~~~and~~~Eq.~(\ref{matchpair6})} \Longrightarrow {\rm Eq.~(\ref{matchpair1})},\;\;{\rm  Eq.~(\ref{matchpair5})~~~and~~~Eq.~(\ref{matchpair6})}\Longrightarrow {\rm Eq.~(\ref{matchpair4})},\\
&&{\rm  Eq.~(\ref{matchpair6})~~~and~~~Eq.~(\ref{matchpair7})}\Longrightarrow {\rm Eq.~(\ref{matchpair2})},\;\;{\rm  Eq.~(\ref{matchpair6})~~~and~~~Eq.~(\ref{matchpair7})}\Longrightarrow {\rm Eq.~(\ref{matchpair3})}.
\vspb
\end{eqnarray*}
\end{rmk}
\vspa
Theorems \ref{thmm1} and \ref{thm1} give the following characterizations of Novikov bialgebras.
\vspa
\begin{cor}\label{Nov-equi}\label{cc1}
Let $(A,\circ)$ be a Novikov algebra and $(A,\Delta)$ a Novikov coalgebra. Let $\cdot$ denote the multiplication on the dual space $A^\ast$ induced by $\Delta$. Then the
following conditions are equivalent.
\begin{enumerate}
\item There is a Manin triple $(A\oplus A^*,(A,\circ),(A^\ast,\cdot))$
of Novikov algebras;
\item $(A,
A^\ast, L_{A, \star}^\ast, -R_A^\ast, L_{A^\ast, \star}^\ast,
-R_{A^\ast}^\ast)$ is a matched pair of Novikov algebras; \item
$(A, \circ, \Delta)$ is a Novikov bialgebra.
\end{enumerate}
\end{cor}
\vspd
\subsection{Manin triples of Novikov algebras and Manin triples of Lie algebras}
\label{ss:manin}

Recall that a bilinear form $(\cdot,\cdot)_L $ on a Lie
algebra $L$ is called {\bf invariant} if
\begin{eqnarray}
( [a, b], c)_L=( a, [b, c])_L \;\; \tforall
a, b, c\in L.
\end{eqnarray}
Then the following statement follows from a routine check.
\begin{pro}\label{Gcorrespond3}
Let $(A, \circ)$ be a Novikov algebra and $(B, \diamond, (\cdot,
\cdot))$ be a quadratic right Novikov algebra. Let
$L=A\otimes B$ be the induced Lie algebra. Suppose that $(A, \circ,
\mathcal{B}(\cdot,\cdot))$ is quadratic.  Define a bilinear form
$(\cdot,\cdot)_L$ on $L$ by
\begin{eqnarray}\label{Lie-bilinear} \notag
( a_1\otimes b_1, a_2\otimes
b_2)_L=\mathcal{B}(a_1,a_2)(b_1, b_2)\;\; \tforall a_1, a_2\in
A, b_1, b_2\in B.
\end{eqnarray}
Then $(\cdot,\cdot)_L$ is a nondegenerate invariant
symmetric bilinear form on the Lie algebra
$L$.
\end{pro}

\begin{defi}\cite{CP}
Let $L$, $L_1$ and $L_2$ be
Lie algebras.  If there is a
nondegenerate invariant symmetric bilinear form
$(\cdot,\cdot)_L$ on $L$ such that
 \begin{itemize}
\item $L_1$ and $L_2$ are Lie subalgebras of
$L$ and $L=L_1\oplus
L_2$ as a direct sum of vector spaces, \item
$L_1$ and $L_2$ are isotropic with respect
to $(\cdot,\cdot)_L$, that is
$(L_i,L_i)_L =0$ for $i=1, 2$,
\end{itemize}
then the triple $(L,L_1, L_2)$ is
called a {\bf Manin triple of Lie algebras}. Two Manin triples
$(L, L_1,L_2)$ and $(L',L_1',L_2')$ are called {\bf isomorphic} if there exists an isomorphism
$\varphi:L\rightarrow L'$ of Lie algebras such that
$\varphi(L_i)=L_i'$ for $i=1,2$ and $(
\varphi(a),\varphi(b))_{L^{'}}=( a,b)_L$ for all $a,b\in
L$.
\end{defi}

Then a straightforward consequence of
Proposition~\ref{Gcorrespond3} is

\begin{pro}\label{Gcorrespond5} Let $(B, \diamond, (\cdot, \cdot))$ be a
quadratic right Novikov algebra. Let $(A, \circ)$ and
$(A^\ast, \bullet)$ be Novikov algebras, and let $A\otimes B$ and $A^\ast\otimes B$ be the induced Lie algebras.  If $(A\oplus
A^\ast, A, A^\ast)$ is a Manin triple of Novikov
algebras, then $(L=(A\oplus A^\ast)\otimes B, A\otimes B,
A^\ast\otimes B)$ is a Manin triple of Lie algebras associated
with the bilinear form
\begin{eqnarray}\label{Manin1}
( a_1\otimes b_1+f\otimes b_2,a_2\otimes b_3+g\otimes
b_4)_L=\langle f, a_2\rangle (b_2, b_3)+\langle g, a_1\rangle
(b_1, b_4),\;\;
\end{eqnarray}
for all $a_1$, $a_2\in A$, $f$, $g\in A^\ast$, and $b_1$, $b_2$,
$b_3$, $b_4\in B$.
\end{pro}

Recall the classical Manin triple characterization of Lie bialgebras.

\begin{pro} \label{equiv-Lie bialgebra} \cite{CP}
Consider a finite-dimensional Lie algebra
$(L, [\cdot, \cdot]_{L})$ and a Lie coalgebra
$(L, \delta)$, so that the linear dual $\delta^*:L^*\otimes L^*\rightarrow L^*$ defines a Lie algebra
$(L^*,[\cdot,\cdot]_{L^*})$. The triple $(L,[\cdot,\cdot],\delta)$ is a Lie bialgebra if and only if
 $(L\oplus L^*,
L, L^*)$ is a Manin triple of Lie algebras associated
with the bilinear form defined by Eq.~\eqref{bilinear}.
\end{pro}

Let $(B, \diamond, (\cdot, \cdot))$ be a quadratic right Novikov
algebra. Through the linear isomorphism $\varphi: B\rightarrow B^\ast$ given by $\langle \varphi(a), b\rangle=(a,b)$ for all $a, b\in B$, a right Novikov algebra $(B^\ast,\diamond')$ is obtained by transporting of structure:
\vspb
\begin{equation}\label{eq:right}
f\diamond' g=\varphi(\varphi^{-1}(f)\diamond \varphi^{-1}(g))\;\;\tforall f,g\in
B^\ast.
\vspa
\end{equation}
Let $(A^\ast, \bullet)$ be a Novikov algebra. Then it
is straightforward to show the induced Lie algebra $
A^\ast\otimes B$ from $(A^\ast, \bullet)$ and $(B,\diamond)$ is
isomorphic to the induced Lie algebra $A^\ast\otimes B^\ast$ from
$(A^\ast, \bullet)$ and $(B^\ast,\diamond')$.
Hence the next conclusion follows immediately.

\begin{lem}\label{cor:equiv}
Assume the conditions in Proposition~\ref{Gcorrespond5} and let the Lie algebra $L^*=A^*\otimes
B^*$ be induced from $(A^\ast, \bullet)$ and $(B^\ast,\diamond')$. Then
$(L\oplus L^*=(A\otimes B)\oplus (A^*\otimes B^*),
L=A\otimes B,L^*=A^*\otimes B^*)$ is a Manin triple of
Lie algebras associated to the bilinear form defined by
Eq.~\eqref{bilinear}. Moreover, this Manin triple is
isomorphic to the Manin triple $((A\oplus A^\ast)\otimes B,
A\otimes B, A^\ast\otimes B)$  associated with the bilinear form
defined by Eq.~\eqref{Manin1}.
\end{lem}

\begin{rmk} \label{rk:quadnov}
The above conclusion gives the reason why $B$  is required to be a
quadratic right Novikov algebra in Theorem \ref{Gcorrespond6}, that is, such a condition guarantees that
 $((A\oplus A^\ast)\otimes B, A\otimes B,
A^\ast\otimes B)$ is isomorphic to $(L\oplus L^*=(A\otimes B)\oplus (A^*\otimes B^*),
L=A\otimes B, L^*=A^*\otimes B^*)$ as Manin triples of Lie algebras.
\end{rmk}

Then it is direct to verify the commutativity of the right square of diagram~\eqref{eq:bigdiag}:
\begin{pro} \label{eq:bialgdouble}
Assume the conditions in Theorem~\ref{Gcorrespond6}, with $(B, \diamond, (\cdot, \cdot))$ being a
quadratic right Novikov algebra. Then the resulting Lie bialgebra
$(L, [\cdot, \cdot], \delta)$ coincides with the one
obtained from the Manin triple $((A\otimes B) \oplus (A^\ast
\otimes B^\ast), A\otimes B, A^\ast\otimes B^\ast)$ of Lie
algebras given in Lemma ~\ref{cor:equiv}, leading to the
commutative diagram$:$ \vspb
$$        \xymatrix{
 \text{$(A, \circ, \Delta)$}\atop\text{a Novikov bialgebra}\ar@{<->}[rr]^-{\rm Cor. \ref{Nov-equi}}  \ar[d]_{\rm Thm. \ref{Gcorrespond6}}&&
\text{$(A\oplus A^\ast, A, A^\ast)$} \atop {\txt{\tiny a Manin triple of  Novikov algebras}} \ar[d]^(.6){\rm Prop. \ref{Gcorrespond5}, ~Lem. \ref{cor:equiv}}\\
 \text{$(L=A\otimes B, [\cdot, \cdot], \delta)$}\atop\text{a Lie bialgebra} \ar@{<->}[rr]^-{\rm Prop. \ref{equiv-Lie bialgebra}}&& \text{$((A\otimes B) \oplus (A^\ast \otimes B^\ast),
A\otimes B, A^\ast\otimes B^\ast)$}\atop\text{a Manin triple of Lie algebras}  }
\vspb
$$
\end{pro}

\subsection{Novikov Yang-Baxter equation, $\mathcal O$-operators of Novikov algebras and \ndend algebras}
\label{ss:nybe}
Let $(A,\circ)$ be a Novikov algebra and $r\in A\otimes A$. In the
following, we consider a special class of Novikov bialgebras $(A,
\circ, \Delta_r)$ when $\Delta_r: A\rightarrow A\otimes A$ is
defined by
\begin{eqnarray}\label{co1}
\Delta_r(a)\coloneqq (L_A(a)\otimes \id+\id\otimes L_{A,\star}(a))r\quad \tforall  a\in A.
\end{eqnarray}
This resembles the coboundary Lie bialgebras \cite{CP}.

Let $r=\sum_\alpha x_\alpha \otimes y_\alpha \in A\otimes A$ and
$r'=\sum_\beta x_\beta '\otimes y_\beta '\in A\otimes A$. Set
$$r_{12}\circ r'_{13}:=\sum_{\alpha,\beta}x_\alpha \circ x_\beta'\otimes y_\alpha\otimes
y_\beta',\;\; r_{12}\circ r'_{23}:=\sum_{\alpha,\beta} x_\alpha\otimes y_\alpha\circ
x_\beta'\otimes y_\beta',\;\;r_{13}\circ r'_{23}:=\sum_{\alpha,\beta} x_\alpha\otimes
x_\beta'\otimes y_\alpha\circ y_\beta',$$
$$r_{13}\circ r'_{12}:=\sum_{\alpha,\beta}x_\alpha\circ x_\beta'\otimes y_\beta'\otimes
y_\alpha,\;\;r_{23}\circ r'_{13}:=\sum_{\alpha,\beta} x_\beta'\otimes x_\alpha\otimes
y_\alpha\circ y_\beta',$$
$$r_{12}\star r'_{23}:=\sum_{\alpha,\beta} x_\alpha\otimes y_\alpha\star
x_\beta'\otimes y_\beta',\;\;r_{13}\star r'_{23}:=\sum_{\alpha,\beta} x_\alpha\otimes
x_\beta'\otimes y_\alpha\star y_\beta'.$$

\begin{lem}\label{coblem1}\label{lem:cob2}
Let $(A,\circ)$ be a Novikov algebra and $r\in A\otimes A$. Define $\Delta_r: A\rightarrow A\otimes A$ by Eq.~\eqref{co1}.
\begin{enumerate}
\item \label{it:coba}$\Delta_r$ satisfies Eq.~{\rm (\ref{Lb5})} if and only if
\begin{equation}
\label{cob4}(\id\otimes (L_A(b\circ a)+L_A(a)L_A(b))+L_{A,\star}(a)\otimes L_{A,\star}(b))(r+\tau r)=0\;\;\tforall  a,b\in A.
\end{equation}
\item \label{it:cobb}
$\Delta_r$ satisfies Eq.~{\rm (\ref{Lb6})} if and only if
\begin{equation}
\label{cob2}
(L_{A,\star}(a)\otimes L_{A,\star}(b)-L_{A,\star}(b)\otimes L_{A,\star}(a))(r+\tau r)=0\;\;\tforall  a,b\in A.
\end{equation}
\item \label{it:cobc}
$\Delta_r$ satisfies Eq.~{\rm (\ref{Lb7})} if and only if
{\small
\begin{equation}
\label{cob3}
\begin{split}
&\Big(-L_{A,\star}(b)\otimes R_A(a)+L_{A,\star}(a)\otimes R_A(b)+R_A(a)\otimes L_A(b)-R_A(b)\otimes L_A(a)+\id\otimes \big(L_A(a)L_A(b)\\
& \ -L_A(b)L_A(a)\big)-\big(L_A(a)L_A(b)-L_A(b)L_A(a)\big)\otimes \id\Big)(r+\tau r)=0 \tforall  a,b\in A.
\end{split}
\end{equation}}
\item \label{it:cobd}
$\Delta_r$ satisfies Eq.~{\rm (\ref{Lc3})} if and only if
{\small
    \begin{eqnarray}
&&\Big(L_A(a)\otimes \id\otimes \id-\id\otimes L_A(a)\otimes \id\Big)\Big((\tau r)_{12}\circ r_{13}+r_{12}\circ r_{23}+r_{13}\star r_{23}\Big)\nonumber\\
&&\hspace{0.3cm}+\Big((\id\otimes L_A(a)\otimes
\id)(r+\tau r)_{12}\Big)\circ r_{23}-\Big((L_A(a)\otimes \id\otimes \id)r_{13}\Big)\circ (r+\tau r)_{12}+\Big(\id\otimes \id\otimes
L_{A,\star}(a)\Big)\label{cob6}\\
&&\hspace{0.3cm}\Big(r_{23}\circ r_{13}-r_{13}\circ
r_{23}-(\id\otimes\id\otimes \id-\tau\otimes \id)(r_{13}\circ
r_{12}+r_{12}\star r_{23})\Big)=0~~\tforall  a\in A.
    \nonumber
\end{eqnarray}}
\item \label{it:cobe}
$\Delta_r$ satisfies Eq.~{\rm (\ref{Lc4})} if and only if
        \begin{equation}
            \label{cob7}(\id\otimes \id\otimes \id-\id\otimes \tau)(\id\otimes \id\otimes
            L_{A,\star}(a))(r_{13}\circ (\tau r)_{23} -r_{12}\star
            r_{23}-r_{13}\circ r_{12})=0\;\;\tforall  a\in A.
\end{equation}
\end{enumerate}
\end{lem}

\begin{proof}
Set $r=\sum_\alpha x_\alpha \otimes y_\alpha$ and let $a, b\in A$. Then
\vspb
$$\Delta_r(a)=\sum_\alpha\Big((a\circ x_\alpha)\otimes y_\alpha+x_\alpha\otimes (a\star
y_\alpha)\Big).
\vspc
$$

\noindent
\eqref{it:coba}. We have
\vspb
\begin{eqnarray*}
&&\Delta_r(b\circ a)-(R_A(a)\otimes \id)\Delta_r(b)-(\id\otimes L_{A,\star}(b))(\Delta_r(a)+\tau\Delta_r(a))\\
&&=\sum_\alpha\Big((b\circ a)\circ x_\alpha \otimes y_\alpha+x_\alpha\otimes (b\circ
a)\star y_\alpha
-(b\circ x_\alpha)\circ a\otimes y_\alpha-x_\alpha\circ a\otimes b\star y_\alpha\\
&&\qquad -a\circ x_\alpha\otimes b\star y_\alpha -x_\alpha\otimes
(a\star y_\alpha)\star b -y_\alpha\otimes b\star (a\circ x_\alpha)
-a\star y_\alpha\otimes b\star x_\alpha\Big)\\
&&=\sum_\alpha\Big((b\circ a)\circ x_\alpha-(b\circ x_\alpha)\circ a)\otimes y_\alpha+x_\alpha\otimes ((b\circ a)\star y_\alpha-(a\star y_\alpha)\star b)\\
&&\qquad -a\star x_\alpha\otimes b\star y_\alpha-y_\alpha\otimes
b\star (a\circ x_\alpha)
-a\star y_\alpha\otimes b\star x_\alpha\Big)\\
&&=\sum_\alpha\Big( -x_\alpha\otimes ((b\circ a)\circ y_\alpha+a\circ (b\circ y_\alpha))-x_\alpha\star a\otimes b\star y_\alpha\\
&&\qquad -y_\alpha\otimes (a\circ (b\circ x_\alpha)+(b\circ
a)\circ x_\alpha)
-a\star y_\alpha\otimes b\star x_\alpha\Big)\\
&&=-\Big(\id\otimes (L_A(b\circ a)+L_A(a)L_A(b))\Big)(r+\tau
r)-\Big(L_{A,\star}(a)\otimes L_{A,\star}(b)\Big)(r+\tau r).
\end{eqnarray*}
Hence Eq.~(\ref{Lb5}) holds if and only if Eq.~(\ref{cob4}) holds.

One similarly verifies Items~\eqref{it:cobb} and \eqref{it:cobc}.

\smallskip

\noindent
\eqref{it:cobd}. We further have
\begin{eqnarray*}
\lefteqn{(\id\otimes \Delta_r)\Delta_r(a)-(\tau\otimes \id)(\id\otimes \Delta_r)\Delta_r(a)-(\Delta_r\otimes \id)\Delta_r(a)+(\tau\otimes \id)(\Delta_r\otimes \id)\Delta_r(a)}\\
&=&\sum_{\alpha,\beta}\Big(a\circ x_\alpha\otimes y_\alpha\circ x_\beta\otimes y_\beta-y_\alpha\circ x_\beta\otimes a\circ x_\alpha\otimes y_\beta+a\circ x_\alpha\otimes x_\beta\otimes y_\alpha\star y_\beta\\
&&\qquad +x_\alpha\otimes (a\star y_\alpha)\circ x_\beta\otimes y_\beta-x_\beta\otimes y_\beta\star (a\circ x_\alpha)\otimes y_\alpha+y_\beta\otimes (a\circ x_\alpha)\circ x_\beta\otimes y_\alpha\\
&&\qquad+x_\alpha\otimes x_\beta\otimes (a\star y_\alpha)\star y_\beta-x_\beta\otimes x_\alpha\otimes (a\star y_\alpha)\star y_\beta-x_\beta\otimes a\circ x_\alpha \otimes y_\alpha\star y_\beta\\
&&\qquad-((a\star y_\alpha)\circ x_\beta\otimes x_\alpha\otimes y_\beta+(a\circ x_\alpha)\circ x_\beta\otimes y_\beta\otimes y_\alpha-(a\circ x_\alpha)\star y_\beta\otimes x_\beta\otimes y_\alpha)\\
&&\qquad+(x_\alpha\star y_\beta\otimes x_\beta\otimes a\star y_\alpha-x_\alpha\circ x_\beta\otimes y_\beta\otimes a\star y_\alpha)\\
&&\qquad-(x_\beta\otimes x_\alpha\star y_\beta\otimes a\star y_\alpha+y_\beta\otimes x_\alpha\circ x_\beta\otimes a\star y_\alpha)\Big)\\
&=&\sum_{\alpha,\beta}\Big(a\circ x_\alpha\otimes y_\alpha\circ x_\beta\otimes y_\beta+a\circ x_\alpha\otimes x_\beta\otimes y_\alpha\star y_\beta+x_\alpha\otimes ((a\star y_\alpha)\circ x_\beta\\
&&\qquad-(a\circ x_\beta)\star y_\alpha)\otimes y_\beta+x_\alpha\otimes x_\beta\otimes ((a\star y_\alpha)\star y_\beta-(a\star y_\beta)\star y_\alpha)\\
&&\qquad-y_\alpha\circ x_\beta\otimes a\circ x_\alpha\otimes y_\beta-x_\beta\otimes a\circ x_\alpha\otimes y_\alpha\star y_\beta\\
&&\qquad-((a\star y_\alpha)\circ x_\beta-y_\alpha\star (a\circ x_\beta))\otimes x_\alpha\otimes y_\beta\\
&&\qquad-(a\circ x_\alpha)\circ x_\beta\otimes y_\beta\otimes
y_\alpha-x_\alpha\circ x_\beta\otimes y_\beta\otimes a\star y_\alpha-x_\beta\otimes x_\alpha\star
y_\beta
\otimes a\star y_\alpha\\
&&\qquad+y_\beta\otimes (a\circ x_\alpha)\circ x_\beta\otimes
y_\alpha+y_\beta\otimes x_\alpha\circ x_\beta\otimes a\star y_\alpha +x_\alpha\star y_\beta\otimes
x_\beta\otimes a\star y_\alpha\Big)\\
&=&\big(L_A(a)\otimes \id\otimes \id-\id\otimes L_A(a)\otimes \id\big)\big((\tau r)_{12}\circ r_{13}+r_{12}\circ r_{23}+r_{13}\star r_{23}\big)\nonumber\\
&&+((\id\otimes L_A(a)\otimes \id)(r+\tau r)_{12})\circ r_{23}-((L_A(a)\otimes \id\otimes \id)r_{13})\circ (r+\tau r)_{12}\nonumber\\
&&+(\id\otimes \id\otimes
L_{A,\star}(a))\big(r_{23}\circ r_{13}-r_{13}\circ
r_{23}-(\id\otimes\id\otimes \id-\tau\otimes \id)(r_{13}\circ
r_{12}+r_{12}\star r_{23})\big).
\end{eqnarray*}
Here we have used the following identities.
\vspb
\begin{eqnarray*}
&&(a\star y_\alpha)\circ x_\beta-(a\circ x_\beta)\star y_\alpha=(a\circ y_\alpha)\circ x_\beta-a\circ (y_\alpha\circ x_\beta),\\
&&(a\star y_\alpha)\star y_\beta-(a\star y_\beta)\star y_\alpha=a\star (y_\beta\circ
y_\alpha)-a\star (y_\alpha\circ y_\beta).
\vspb
\end{eqnarray*}
Hence Eq.~(\ref{Lc3}) holds if and only if Eq.~(\ref{cob6}) holds.  One similarly verifies Item \eqref{it:cobe}.
\end{proof}
\vspb
By Lemma~\ref{coblem1}, we arrive at the following conclusion.
\vspb
\begin{thm}\label{thmco1}
Let $(A,\circ)$ be a Novikov algebra and $r\in A\otimes A$. Define
$\Delta_r: A\rightarrow A\otimes A$ by Eq.~{\rm (\ref{co1})}. Then
 $(A,\circ, \Delta_r)$ is a Novikov bialgebra if and
only if Eqs.~(\ref{cob4})-(\ref{cob7}) hold.
\vspb
\end{thm}

The we have the following special case.
\vspb
\begin{cor}\label{corco2}
Let $(A,\circ)$ be a Novikov algebra and $r\in A\otimes A$ be
skewsymmetric. Define $\Delta_r: A\rightarrow A\otimes A$ by
Eq.~{\rm (\ref{co1})}. Then
 $(A,\circ, \Delta_r)$ is a Novikov bialgebra if and only if the following
equalities hold.
{\small
\begin{eqnarray}
&&(L_A(a)\otimes \id\otimes \id-\id\otimes L_A(a)\otimes
\id)(\id\otimes \tau) (r\diamond r) +(\id\otimes
\id\otimes L_{A,\star}(a))(r\diamond r-(\tau\otimes \id)(r\diamond r))=0,\label{cob8}\\
&&\label{cob9}(\id\otimes \id\otimes \id- \id\otimes \tau)(\id\otimes
\id\otimes L_{A,\star}(a))(r\diamond r)=0\;\;\tforall  a\in A,
\end{eqnarray}}
where
\vspd
\begin{eqnarray} \notag
r\diamond r\coloneqq r_{13}\circ r_{23} +r_{12}\star r_{23}+r_{13}\circ
r_{12}.
\vspb
\end{eqnarray}
In particular, if $r\diamond r=0$, then  $(A,\circ, \Delta_r)$ is
a Novikov bialgebra.
\end{cor}

\begin{proof}
By the skewsymmetry of
$r$, we obtain that Eq.~(\ref{cob6}) holds if and only if
Eq.~(\ref{cob8}) holds, and Eq.~(\ref{cob7}) holds if and only
if Eq.~(\ref{cob9}) holds.
Then the conclusion follows.
\vspb
\end{proof}

Corollary~\ref{corco2} motivates the following notion similar to
the classical Yang-Baxter equation (CYBE) for Lie algebras
~\cite{CP}.
\begin{defi}
Let $(A,\circ)$ be a Novikov algebra and $r\in A\otimes A$. The
equation
\vspb
$$r\diamond r\coloneqq r_{13}\circ r_{23} +r_{12}\star r_{23}+r_{13}\circ
r_{12}=0
\vspb
$$
is called the {\bf Novikov Yang-Baxter equation (NYBE)} in $A$.
\end{defi}

To extend the close relationship between skewsymmetric
solutions of the CYBE and quasi-Frobenius Lie algebras (see
\cite{BFS}) to the context of Novikov algebras, we also give the
following notion.
\begin{defi}
Let $(A, \circ)$ be a Novikov algebra. If there is a skewsymmetric nondegenerate bilinear form $\omega(\cdot,\cdot)$ on $A$ satisfying
\begin{eqnarray}
&&\omega(a\circ b,c)-\omega(a\star c, b)+\omega(c\circ
b,a)=0\;\;\text{for all $a$, $b$, $c\in A$,}
\end{eqnarray}
then $(A, \circ, \omega(\cdot,\cdot))$ is called a {\bf quasi-Frobenius Novikov algebra}.
\end{defi}

Thus we obtain the following relation.
\begin{pro}\label{quasi-Nov-equi}
Let $(A, \circ)$ be a Novikov algebra with a nondegenerate
bilinear form $\omega(\cdot,\cdot)$. Let $\{e_\alpha |\alpha\in
I\}$ be a basis of $A$ and $\{f_\alpha |\alpha\in I\}$ be its dual
basis associated with the bilinear form
$\omega(\cdot,\cdot)$, and $r=\sum_{\alpha\in
I}e_\alpha\otimes f_\alpha\in A\otimes A$. Then $r$ is a
skewsymmetric solution of the NYBE in $A$ if and only if $(A,
\circ, \omega(\cdot,\cdot))$ is a quasi-Frobenius Novikov algebra.
\end{pro}
\begin{proof}
The proof follows from the same argument as the one for Lie algebras, as presented in~\cite[Theorem 3.1]{BFS}.
\end{proof}

For a finite-dimensional vector space $A$, the isomorphism
\vspb
$$A\ot A\cong
\Hom_{\bf k}(A^*,{\bf k})\ot A\cong \Hom_{\bf k}(A^*,A)
\vspa
$$
identifies an $r\in A\otimes
A$ with a map from $A^*$ to $A$ which we denote by $T^r$.
Explicitly, writing $r=\sum_{\alpha}x_\alpha\otimes y_\alpha$, then
\vspe
\begin{equation}\label{eq:4.12} \notag
T^r:A^*\to A, \quad T^r(f)=\sum_{\alpha}\langle f, x_\alpha\rangle y_\alpha
~~\tforall  f\in A^*.
\vspd
\end{equation}
A routine check gives the following property.
\begin{thm}\label{operator1}
Let $(A, \circ)$ be a Novikov algebra and $r\in A\otimes A$ be
skewsymmetric. Then $r$ is a solution of the \nybe in $(A,\circ)$
if and only if $T^r$ satisfies
\vspa
\begin{eqnarray}\label{operr5} \notag
T^r(f)\circ
T^r(g)=T^r(L_{A,\star}^\ast(T^r(f))g)-T^r(R_A^\ast(T^r(g))f)~~~\tforall ~~f,~~g\in
A^\ast.
\vspd
\end{eqnarray}
\end{thm}

Theorem \ref{operator1} shows that $T^r$ plays a role similar to
that of the classical $\calo$-operator of a Lie algebra, as the operator form of the CYBE~\cite{Ku1}.  This
motivates us to give the following notion, as the operator form of
the \nybe.
\begin{defi}
Let $(A,\circ)$ be a Novikov algebra and $(V, l_A, r_A)$ be a
representation. A linear map $T: V\rightarrow A$ is called an {\bf $\mathcal{O}$-operator} of $(A,\circ)$ associated to $(V, l_A,r_A)$ if $T$ satisfies
\vspb
\begin{eqnarray*}
T(u)\circ T(v)=T(l_A(T(u))v)+T(r_A(T(v))u)~~~~~\tforall ~~u,~~v\in
V.
\end{eqnarray*}
\end{defi}

Hence for a Novikov algebra $(A,\circ)$ and a skewsymmetric element $r\in A\otimes A$, Theorem~\ref{operator1} implies that $r$ is a
solution of the \nybe in $(A,\circ)$ if and
only if $T^r$ is an $\mathcal O$-operator of $(A,\circ)$
associated to the representation $(A^*,L_{A,\star}^*,-R_A^*)$.

\begin{thm}\label{Goper}
Let $(A,\circ)$ be a Novikov algebra and $(V, l_A, r_A)$ be a
representation. Let $T: V\rightarrow A$ be a linear map which is
identified with $r_T\in A\otimes V^\ast \subseteq
(A\ltimes_{l_A^\ast+r_A^\ast,-r_A^\ast} V^\ast) \otimes
(A\ltimes_{l_A^\ast+r_A^\ast,-r_A^\ast} V^\ast)$ through
 ${\rm Hom}_{\bf k}(V, A)\cong A\otimes V^\ast$.
Then $r=r_T-\tau r_T$ is a solution of the \nybe in the Novikov algebra
$(A\ltimes_{l_A^\ast+r_A^\ast,-r_A^\ast} V^\ast, \bullet)$
in Proposition~\ref{pro:semi}
if and only if $T$ is an
$\mathcal{O}$-operator of $(A,\circ)$ associated to $(V, l_A,
r_A)$.
\end{thm}

\begin{proof}
The proof follows the same argument as the one in \cite[Section 2]{Bai} for Lie algebras.
\end{proof}

\begin{defi}
A {\bf \ndend algebra} is a triple $(A,\lhd,\rhd)$, where $A$ is a vector space, and $\lhd$ and $\rhd$ are binary operations such that
\vspb
\begin{eqnarray}
&&\label{ND1}x\rhd (y\rhd z)=(x\rhd y+x\lhd y)\rhd z+y\rhd (x\rhd z)-(y\rhd x+y\lhd x)\rhd z,\\
&&\label{ND2} x\rhd (y\lhd z)=(x\rhd y)\lhd z+y\lhd (x\lhd z+x\rhd z)-(y\lhd x)\lhd z,\\
&&\label{ND3} (x\lhd y+x\rhd y)\rhd z=(x\rhd z)\lhd y,\\
&&\label{ND4} (x\lhd y)\lhd z=(x\lhd z)\lhd y \tforall x, y,z\in A.
\vspd
\end{eqnarray}
\end{defi}

\begin{rmk}
The operad of pre-Novikov algebras is the
successor of the operad of Novikov algebra in the sense
of~\cite{BBGN}.
On the other hand, if $\lhd$ and $\rhd$ only satisfy
Eqs.~(\ref{ND1}) and (\ref{ND2}), then $(A, \lhd, \rhd)$ is
called an {\bf $L$-dendriform algebra} in \cite{Bai-Liu-Ni}. \end{rmk}
\vspb
\begin{ex}
Recall~\cite{Lod} that a {\bf Zinbiel algebra} $(A,\cdot)$ is a
vector space $A$ with a binary operation $\cdot: A\otimes
A\rightarrow A$ satisfying
\vspa
$$
a\cdot (b\cdot c)=(b\cdot a)\cdot c+(a\cdot b)\cdot c\;\;\tforall
a,b,c\in A.
\vspb
$$
For a derivation $D$ on a Zinbiel algebra $(A,\cdot)$, define binary operations  $\lhd$ and $\rhd: A\otimes
A\rightarrow A$ by
\vspb
$$
a\lhd b\coloneqq D(b)\cdot a,\;\;a\rhd b\coloneqq a\cdot D(b)\;\;\tforall  a,b\in A.
\vspb
$$
A direct check shows that $(A,\lhd,\rhd)$ is a \ndend algebra.
\vspb
\end{ex}

For a \ndend algebra $(A, \lhd, \rhd)$, define linear maps $L_{\rhd},
R_{\lhd}:A\rightarrow {\rm End}_{\bf k}(A)$ by
\vspa
\begin{eqnarray*}
L_{\rhd}(a)(b)\coloneqq a\rhd b, \quad
R_{\lhd}(a)(b)\coloneqq b\lhd a~~\;\tforall  a, b\in A.
\vspb
\end{eqnarray*}

\begin{pro}\label{pro:ndend}
Let $(A, \lhd, \rhd)$ be a \ndend algebra. The binary operation
\vspb
\begin{eqnarray}
\label{ND5}
\circ: A\otimes A\rightarrow A, \quad
x \circ y\coloneqq x\lhd y+x\rhd y~~~~\;\;\tforall  x, y\in A,
\vspb
\end{eqnarray}
defines a Novikov algebra, which is called the {\bf associated Novikov algebra} of $(A, \lhd, \rhd)$.
Moreover,  $(A, L_{\rhd}, R_{\lhd})$ is
a representation of $(A, \circ)$. Conversely, let $A$ be a vector space
with binary operations $\rhd$ and $\lhd$. If $(A, \circ)$ defined by Eq.~\eqref{ND5} is a Novikov
algebra and $(A, L_{\rhd}, R_{\lhd})$ is a representation of $(A,
\circ)$, then $(A, \lhd, \rhd)$ is a \ndend algebra.
\end{pro}
\vspb
\begin{proof}
The first statement follows from~\cite{BBGN} and the other statements are easily verified.
\end{proof}
\vspb
The following conclusion establishes the relationship between
\ndend algebras and the $\mathcal O$-operators of the associated
Novikov algebras.
\vspa
\begin{pro}\label{cor:iden}
\begin{enumerate}
\item \label{it:1}Let $(A, \circ)$ be a Novikov algebra, $(V, l_A,
r_A)$ be a representation of $(A, \circ)$ and $T: V\rightarrow A$ be an
$\mathcal{O}$-operator associated to $(V, l_A, r_A)$. Then there
exists a \ndend algebra structure on $V$ defined by
\vspa
\begin{eqnarray}\label{eq:ndend} \notag
u\rhd v=l_A(T(u))v,\;\; u\lhd v= r_A(T(v))u \tforall u,~~v\in V.
\end{eqnarray}
\item\label{it:2}
Let $(A, \lhd, \rhd)$ be a \ndend algebra and $(A, \circ)$ be the
associated Novikov algebra. Then the identity map is an $\mathcal
O$-operator of $(A, \circ)$ associated to the representation $(A,
L_{\rhd}, R_{\lhd})$.
\end{enumerate}
\end{pro}

\begin{proof}
(\ref{it:1}) follows from a direct checking and
(\ref{it:2}) follows from Proposition~\ref{pro:ndend}.
\end{proof}

\begin{thm}\label{NYB-ND}
Let $(A, \lhd, \rhd)$ be a \ndend algebra and $(A,
\circ)$ be the associated Novikov algebra.  Then
\vspc
\begin{eqnarray}\label{eq:solu}
r:=\sum_{\alpha=1}^n(e_\alpha\otimes e_\alpha^\ast-e_\alpha^\ast\otimes e_\alpha),
\vspa
\end{eqnarray}
is a skewsymmetric solution of the \nybe in
the Novikov algebra $A\ltimes_{L_\rhd^\ast+R_\lhd^\ast,
-R_\lhd^\ast}A^\ast$, where $\{e_1, \ldots, e_n\}$ is a linear basis of
$A$ and $\{e_1^\ast, \ldots, e_n^\ast\}$ is the dual basis of
$A^\ast$. Moreover, $A\ltimes_{L_\rhd^\ast+R_\lhd^\ast,
    -R_\lhd^\ast}A^\ast$ with the bilinear form $\omega(\cdot,\cdot)$ given by
\begin{eqnarray}\label{skewsymm-bin}
\omega(a+f, b+g)=\langle g, a\rangle-\langle f, b\rangle \;\; \text{for all $a$, $b\in A$, $f$, $g\in
A^\ast$,}
\end{eqnarray} is a quasi-Frobenius Novikov algebra.
\end{thm}

\begin{proof}
By Proposition~\ref{cor:iden} (\ref{it:2}), the identity map
$\id: A\rightarrow A$ is an $\mathcal{O}$-operator of $(A, \circ)$
associated to $(A, L_\rhd, R_\lhd)$.  Hence the first
conclusion follows from Proposition~\ref{Goper}. Moreover, in the
vector space $A\oplus A^*$, the dual basis of the basis $\{e_1,
\cdots, e_n, e_1^\ast, \cdots, e_n^\ast\}$ associated with
$\omega(\cdot, \cdot)$ is $\{e_1^\ast, \cdots, e_n^\ast, -e_1,
\cdots, -e_n\}$. Hence the second conclusion follows from
Proposition \ref{quasi-Nov-equi}.
\end{proof}

\vspd
\section{Infinite-dimensional Lie bialgebras from the Novikov Yang-Baxter equation}
\label{sec:novlie} In this section, we use skewsymmetric
solutions of the NYBE to construct skewsymmetric solutions of the CYBE. We also present a construction of quasi-Frobenius $\ZZ$-graded Lie algebras from quasi-Frobenius Novikov algebras corresponding to a class of skewsymmetric solutions of the NYBE. Along with the results in Section~\ref{ss:nybe},
\ndend algebras can be used to obtain a large supply
of infinite-dimensional Lie bialgebras and quasi-Frobenius $\ZZ$-graded Lie algebras. An explicit example is provided.

In this section, $(A, \circ)$ is assumed to be a finite-dimensional Novikov algebra.

Let $L=\oplus_{i\in \ZZ} L_i$ be a $\ZZ$-graded Lie algebra.
Suppose that $r=\sum_{i,j,\alpha}a_{i\alpha}\otimes b_{j\alpha}\in L\hatot L$ as in Eq.~\eqref{eq:ssum}.
We denote
\vspb
\begin{eqnarray*}
&&[r_{12},r_{13}]:=\!\sum_{i,j,k,l,\alpha, \beta}[a_{i\alpha}, a_{k\beta}]\otimes b_{j\alpha}\otimes b_{l\beta},\;[r_{12},r_{23}]:=\!\sum_{i,j,k,l,\alpha, \beta}a_{i\alpha}\otimes [b_{j\alpha},a_{k\beta}]\otimes b_{l\beta},\\
&& [r_{13}, r_{23}]:=\!\sum_{i,j,k,l,\alpha,\beta} a_{i\alpha}\otimes a_{k\beta}\otimes [b_{j\alpha}, b_{l\beta}],
\vspb
\end{eqnarray*}
provided the sums make sense. Note that these sums make sense when $r=\sum_{i\in \ZZ, \alpha} c_{i,\alpha}\otimes d_{-i-s,\alpha}\in L\widehat{\otimes} L $ for some fixed $s\in \ZZ$, which is the case that we are interested in next.

If $r\in L\hatot L$ is skewsymmetric and satisfies the {\bf classical Yang-Baxter equation (CYBE)}
\vspb
$$[r_{12}, r_{13}]+[r_{12}, r_{23}]+[r_{13}, r_{23}]=0
\vspa
$$
as an element in $L \hatot L \hatot L$, then $r$ is called a {\bf completed solution of the CYBE} in $L$.
\vspb
\begin{pro}\label{Lie-coboundary}
If $r\in L\hatot L$ is a skewsymmetric completed
solution of the CYBE in $L$, then for the linear map $\delta:
L\rightarrow L\hatot L$ defined
by
\vspa
\begin{equation} \label{eq:rdelta}
\delta
(x):=(\ad_x\hatot  \id+\id \hatot  \ad_x)r \tforall x\in L,
\vspa
\end{equation}
the triple $(L,[\cdot,\cdot],\delta)$ is a \complete Lie bialgebra.
\vspb
\end{pro}
\begin{proof}
The result holds when $r\in L\ot L$ is a
skewsymmetric solution of the CYBE~\cite{CP}. The same
argument extends to the completed case.
\end{proof}

We give the following relation between the solutions of the NYBE and those of the CYBE.

\begin{pro}\label{pro:NYBE}
Let $(A,\circ)$ be a Novikov algebra and $\big(B=\oplus_{i\in
\ZZ}B_i,\diamond, (\cdot, \cdot)\big)$ be a quadratic $\ZZ$-graded
right  Novikov algebra. Let $L=A\otimes B$ be the induced Lie
algebra. Suppose that $r=\sum_\alpha x_\alpha\otimes y_\alpha\in
A\otimes A$ is a skewsymmetric solution of the NYBE in $A$.
Then for a basis $\{e_p\}_{p\in \Pi}$ consisting of homogeneous
elements of $B$ and its homogeneous dual basis $\{f_p\}_{p\in
\Pi}$ associated with the bilinear form $(\cdot, \cdot)$, the
tensor element
\vspa
\begin{eqnarray}\label{eq:N-CYBE}
r_L\coloneqq \sum_{p\in \Pi} \sum_\alpha
(x_\alpha\otimes e_p)\otimes (y_\alpha\otimes f_p) \in
L \hatot L
\vspb
\end{eqnarray}
is a skewsymmetric completed solution of the CYBE in $L$. Furthermore, if the quadratic $\mathbb{Z}$-graded right Novikov algebra is
$(B, \diamond, (\cdot,\cdot))$ $=({\bf k}[t,t^{-1}], \diamond,
(\cdot, \cdot))$ from Example~\ref{Laurent-Bilinear}, then
\vspa
\begin{eqnarray}\label{eq:affine} 
	r_L:=\sum_{ i\in
\mathbb{Z}}\sum_\alpha x_\alpha t^i\,\otimes \,y_\alpha t^{-i-1}\in L
\hatot  L
\vspb
\end{eqnarray}
is a skewsymmetric completed solution of the \cybe in $L$ if and
only if $r$ is a skewsymmetric solution of the \nybe in $A$.
\end{pro}
If $B$ is a finite-dimensional right Novikov algebra, then Eq. (\ref{eq:N-CYBE}) is a finite sum.

\begin{proof} Adopting the notation in Eq.~\eqref{eq:pairb}, we have
\vspa
\begin{eqnarray*}
\Big(e_q\otimes e_s, \sum_{p\in \Pi}e_p\otimes
f_p\Big)=\sum_{p\in \Pi}(e_q, e_p)(e_s,f_p)=(e_q, e_s).
\vspb
\end{eqnarray*}
Since $(\cdot, \cdot)$ on $B$ is symmetric and nondegenerate, we obtain
$\sum_{p\in \Pi}e_p\otimes f_p=\sum_{p\in
\Pi}f_p\otimes e_p$. Therefore, $r_L$ is
skewsymmetric. Furthermore,
\vspb
\begin{eqnarray*}
&&[{r_L}_{12},{r_L}_{13}]+[{r_L}_{12}, {r_L}_{23}]+[{r_L}_{13}, {r_L}_{23}]\\
&=&\sum_{p,q\in \Pi}\sum_{\alpha,\beta} \Big((x_\alpha\circ x_\beta\otimes e_p\diamond e_q-x_\beta\circ x_\alpha\otimes e_q\diamond e_p)\otimes  (y_\alpha\otimes f_p)\otimes (y_\beta\otimes f_q)\\
&&+(x_\alpha\otimes e_p)\otimes (y_\alpha\circ x_\beta\otimes f_p\diamond e_q-x_\beta\circ y_\alpha\otimes e_q\diamond f_p)\otimes (y_\beta\otimes f_q)\\
&&+(x_\alpha\otimes e_p)\otimes (x_\beta\otimes
e_q)\otimes (y_\alpha\circ y_\beta\otimes f_p\diamond f_q-y_\beta\circ
y_\alpha\otimes f_q\diamond f_p)\Big).
\vspb
\end{eqnarray*}
For $s$, $u$, $v\in \Pi$, adopting the notation in Eq.~\eqref{eq:pairb} we obtain
\vspa
{\small \begin{eqnarray*}
&\Big(e_s\otimes e_u\otimes e_v, \sum\limits_{p,q\in \Pi} e_p\diamond e_q\otimes f_p\otimes f_q\Big)=(e_s, e_u\diamond e_v),\;\;\Big(e_s\otimes e_u\otimes e_v, \sum\limits_{p,q\in \Pi} e_q\diamond e_p\otimes f_p\otimes f_q\Big)=(e_s, e_v\diamond e_u),& \\
&\Big(e_s\otimes e_u\otimes e_v, \sum\limits_{p,q\in \Pi} e_p\otimes
f_p\ast e_q\otimes f_q\Big)=-(e_s, e_u\diamond e_v+e_v\diamond
e_u).&
\end{eqnarray*}}
Therefore, by the nondegeneracy of $(\cdot,\cdot)$ on $B$,
we get
\begin{eqnarray*}
\sum_{p,q\in \Pi}e_p\otimes f_p\diamond e_q\otimes
f_q=-\sum_{p,q\in \Pi}(e_p\diamond e_q\otimes  f_p\otimes
f_q+e_q\diamond e_p\otimes f_p\otimes f_q).
\vspb
\end{eqnarray*}
Similarly, we obtain
\vspa
\begin{eqnarray*}
&&\sum_{p,q\in \Pi} e_p\otimes e_q\otimes f_p\diamond f_q=-\sum_{p,q\in \Pi}(e_p\otimes e_q\diamond f_p\otimes f_q+e_p\otimes e_q\otimes f_q\diamond f_p),\\
&&\sum_{p,q\in \Pi}e_p\otimes e_q\diamond f_p\otimes
f_q=\sum_{p,q\in \Pi} e_q\diamond e_p\otimes  f_p\otimes
f_q,~~\sum_{p,q\in \Pi} e_p\otimes e_q\otimes
f_q\diamond f_p=\sum_{p,q\in \Pi} e_p\diamond e_q\otimes
f_p\otimes f_q.
\vspb
\end{eqnarray*}
Then we have
{\small
\begin{eqnarray*}
&&[{r_L}_{12},{r_L}_{13}]+[{r_L}_{12}, {r_L}_{23}]+[{r_L}_{13}, {r_L}_{23}]\\
&=&\sum_{p,q\in \Pi} \sum_{\alpha,\beta}\Big((x_\alpha\circ x_\beta\otimes e_p\diamond e_q)\otimes (y_\alpha\otimes f_p)\otimes (y_\beta\otimes f_q)-(x_\alpha\otimes e_p\diamond e_q)\otimes (y_\alpha\circ x_\beta \otimes f_p)\otimes (y_\beta\otimes f_q)\\
&&-(x_\alpha\otimes e_p\diamond e_q)\otimes (x_\beta\otimes
f_p)\otimes (y_\alpha\circ y_\beta\otimes f_q)
-(x_\alpha\otimes e_p\diamond e_q)\otimes (x_\beta \otimes f_p)\otimes (y_\beta\circ y_\alpha\otimes f_q\diamond f_p)\Big)\\
&&-\sum_{p, q\in \Pi}\sum_{\alpha,\beta}\Big((x_\beta\circ x_\alpha\otimes e_q\diamond e_p)\otimes (y_\alpha\otimes f_p)\otimes (y_\beta\otimes f_q)+(x_\alpha\otimes e_q\diamond e_p)\otimes(y_\alpha\circ x_\beta\otimes f_p)\otimes (y_\beta\otimes f_q)\\
&&+(x_\alpha\otimes e_q\diamond e_p)\otimes (x_\beta\circ y_\alpha\otimes
f_p)\otimes (y_\beta\otimes f_q)
+(x_\alpha\otimes e_q\diamond e_p)\otimes (x_\beta\otimes f_p)\otimes (y_\alpha\circ y_\beta \otimes f_q)\Big)\\
&=& 0.
\end{eqnarray*}
} Therefore, $r_L$ is a skewsymmetric completed
solution of the CYBE in $L$.

Next, we consider the case when $(B, \diamond, (\cdot, \cdot))=({\bf k}[t,t^{-1}],
\diamond, (\cdot, \cdot))$. Note that $\{t^{-i-1}|i\in \mathbb{Z}\}$ is the basis of ${\bf k}[t, t^{-1}]$ dual to $\{t^{i}|i\in
\mathbb{Z}\}$ associated with the bilinear form defined by $(t^i, t^j)=\delta_{i+j+1,0}$ for
all $i$, $j\in \mathbb{Z}$. Then the ``if" part follows from the
proof above.

Since $\widehat{\tau }r_L=\sum_{ i\in
\mathbb{Z}}\sum_\alpha y_\alpha t^{-i-1}\otimes x_\alpha t^{i} =\sum_{
j\in \mathbb{Z}} \sum_\alpha y_\alpha t^{j}\otimes x_\alpha t^{-j-1}$, we find
that $r$ is skewsymmetric by setting $i=0$. Moreover,
{\small
\begin{eqnarray*} 0&=&[{r_L}_{12},
{r_L}_{13}]
+[{r_L}_{12},{r_L}_{23}]+[{r_L}_{13},{r_L}_{23}]\\
&=&\sum_{i,j\in \mathbb{Z}}\sum_{\alpha,\beta} (i(x_\alpha\circ
x_\beta)t^{i+j-1}-j(x_\beta\circ x_\alpha)t^{i+j-1})\otimes
y_\alpha t^{-i-1}\otimes y_\beta t^{-j-1}+x_\alpha t^i\otimes ((-i-1)(y_\alpha \circ x_\beta)t^{-i+j-2}\\
&&\hspace{0.4cm}-j(x_\beta\circ y_\alpha)t^{-i+j-2})\otimes
y_\beta t^{-j-1}+x_\alpha t^i\otimes x_\beta t^j\otimes
((-i-1)(y_\alpha\circ y_\beta)t^{-i-j-3}+(j+1)(y_\beta\circ y_\alpha)t^{-i-j-3}).
\end{eqnarray*}}
Comparing the coefficients of $1\otimes t^{-1}\otimes t^{-2}$
yields that $r$ is a skewsymmetric solution of the \nybe in
the Novikov algebra $A$.
\vspb
\end{proof}
\begin{rmk}
Let $(B, \diamond)=({\bf k}[t,t^{-1}], \diamond, (\cdot,
\cdot))$ be the quadratic $\mathbb{Z}$-graded right Novikov
algebra in Example \ref{Laurent-Bilinear}. Set
$r_L=r(t_1,t_2)$ as in Eq.~\eqref{eq:affine}. The
CYBE for $r_L$ can
be written as
\begin{eqnarray*}
[r_{12}(t_1,t_2),r_{13}(t_1,t_3)]+[r_{12}(t_1,t_2),r_{23}(t_2,t_3)]+[r_{13}(t_1,t_3),
r_{23}(t_2,t_3)]=0.
\end{eqnarray*}
Although the form of this equation resembles the classical
Yang-Baxter equation \emph{with spectral parameters}
\cite{CP,KZ}, its meaning is different, due to the fact that it does
not make sense when $t_1$, $t_2$, $t_3$ are taken to be in ${\bf
k}$.
\end{rmk}
\begin{ex}
Let $(A={\bf k} a\oplus {\bf k} b \oplus {\bf k}c, \circ)$ be the Novikov algebra given by
\begin{eqnarray*}
&&a\circ a=a,~~ a\circ b=\frac{1}{2}b,~~b\circ a=b,~~a\circ c=0,~~c\circ a=c,\\
&&b\circ b=c,~~b\circ c=c\circ b=c\circ c=0.
\end{eqnarray*}
Then the affinization $L$ of $(A, \circ)$ is the vector space spanned by $\left\{a_i:=at^i, b_i:=bt^i, c_i:=ct^i\,\big |\, i\in \ZZ\right\}$ with the Lie brackets
\begin{eqnarray*}
&&[a_i, a_j]=(i-j)a_{i+j-1},\;\;[a_i, b_j]=(\frac{i}{2}-j)b_{i+j-1},\;\;[a_i, c_j]=-jc_{i+j-1},\\
&&[b_i, b_j]=(i-j)c_{i+j-1},\;\;[b_i, c_j]=[c_i,c_j]=0\;\;\text{for all $i$, $j\in \ZZ$.}
\end{eqnarray*}
Note that $L$ is isomorphic to the centerless Schr\"odinger-Virasoro algebra given in \cite{H}.
\\
\indent Let $r=b\otimes c-c\otimes b$. It is easy to check that
$r$ is a skewsymmetric solution of the NYBE in the Novikov algebra $(A,\circ)$. Then by
Proposition \ref{pro:NYBE}, $r_L=\sum_{i\in \ZZ}(b_i\otimes
c_{-i-1}-c_i\otimes b_{-i-1})$ is a skewsymmetric completed
solution of the CYBE in the Lie algebra $L$.
\vspd
\end{ex}

\begin{cor}\label{cor:same}
Under the same assumption as in Proposition~\ref{pro:NYBE},
 let
$\Delta_r:A\rightarrow A\otimes A$ be a linear map defined by
Eq.~\eqref{co1} with $r\in A\otimes A$ and $\delta:L\rightarrow L\hatot L$ be a linear map defined by
Eq.~\eqref{GLiebi1}. Then $(A,\circ,-\Delta_r)$ is a Novikov
bialgebra and hence $(L,[\cdot,\cdot],\delta)$ is a
\complete Lie bialgebra by Theorem~\ref{Gcorrespond6}. It
coincides with the \complete Lie bialgebra with $\delta$ defined
by Eq.~\eqref{eq:rdelta} through $r_{L}$ by
Proposition~\ref{Lie-coboundary}, where $r_{L}$ is defined
by Eq.~\eqref{eq:N-CYBE}. In other words, the second square from
the right in the diagram~\eqref{eq:bigdiag} commutes$:$ \vspa
$$  \xymatrix{
\text{skewsymmetric solutions}\atop \text{of NYBE} \ar[rr]^-{\rm Cor. \ref{corco2}}   \ar[d]_-{\rm Prop. \ref{pro:NYBE}}&& \text{Novikov}\atop \text{ bialgebras} \ar[d]_-{\rm Thm. \ref{Gcorrespond6}}\\
\text{skewsymmetric solutions}\atop \text{of CYBE} \ar[rr]^-{\rm
Prop. \ref{Lie-coboundary}}
            && \text{Lie}\atop \text{ bialgebras}  }
\vspb
$$
\end{cor}

\begin{proof}
By Corollary \ref{corco2}, $(A, \circ, \Delta_r)$ is a Novikov
bialgebra. Then $(A, \circ, -\Delta_r)$ is also a Novikov
bialgebra.
By Theorem \ref{Gcorrespond6}, there is a \complete Lie
bialgebra structure $(L, [\cdot, \cdot], \delta)$ on
$L$ where $\delta$ is induced from $-\Delta_r$ by Eq.
(\ref{GLiebi1}). That is, for all $a\in A$ and $b\in B$, we have
\vspa
\begin{eqnarray*}
\delta(a\otimes b)&=& \sum_{i,i,\beta}\sum_\alpha(-(a\circ x_\alpha\otimes b_{1i\beta})\otimes (y_\alpha\otimes b_{2j\beta})+(y_\alpha\otimes b_{2j\beta})\otimes( a\circ x_\alpha\otimes b_{1i\beta})\\
&&-(x_\alpha\otimes b_{1i\beta})\otimes (a\star y_\alpha\otimes
b_{2j\beta})+(a\star y_\alpha\otimes b_{2j\beta})\otimes (x_\alpha\otimes
b_{1i\beta})).
\vspc
\end{eqnarray*}
On the other hand,
\vspb
{\small \begin{eqnarray*}
&&(\ad_{a\otimes b}\hatot  \id+\id \hatot  \ad_{a\otimes b})\sum_{p\in \Pi}\sum_\alpha (x_\alpha\otimes e_p)\otimes (y_\alpha \otimes f_p)\\
&=&\sum_{p\in \Pi}\sum_\alpha\Big((a\circ x_\alpha\otimes b\diamond e_p-x_\alpha\circ
a\otimes e_p\diamond b)\otimes (y_\alpha\otimes f_p) \\
&&\qquad \quad +(x_\alpha\otimes
e_p)\otimes (a\circ y_\alpha\otimes b\diamond f_p-y_\alpha\circ
a\otimes f_p\diamond b)\Big).
\vspc
\end{eqnarray*}}
For the basis elements $e_q$, $e_s\in B$, we have
\begin{eqnarray*}
\Big(e_q\otimes e_s, \sum_{p\in \Pi} b\diamond e_p\otimes
f_p\Big) &=&(e_q,b\diamond e_s)=(-e_q\diamond e_s-e_s\diamond e_q, b).
\end{eqnarray*}
Similarly, we obtain
\begin{eqnarray*}
\Big(e_q\otimes e_s, \sum_{i,j,\beta} b_{1i\beta}\otimes b_{2j\beta}\Big)=(e_q\diamond
e_s, b),~~~\Big(e_q\otimes e_s, \sum_{p\in \Pi} e_p\otimes
f_p\diamond b\Big)=(e_q\diamond e_s,b).
\end{eqnarray*}
Then by the nondegeneracy of $(\cdot, \cdot)$, we
have
\begin{eqnarray*}
&&\sum_{p\in \Pi} b\diamond e_p\otimes f_p=\sum_{p\in \Pi}b\diamond f_p\otimes e_p=-\sum_{i,j,\beta}(b_{1i\beta}\otimes b_{2j\beta}+b_{2j\beta}\otimes b_{1i\beta}),\\
&&\sum_{p\in \Pi}e_p\otimes f_p\diamond b=\sum_{i,j,\beta}
b_{1i\beta}\otimes b_{2j\beta},~~~\sum_{p\in \Pi}e_p\diamond
b\otimes f_p=\sum_{i,j,\beta} b_{2j\beta}\otimes b_{1i\beta}.
\end{eqnarray*}
Therefore,
{\small
\begin{eqnarray*}
&&(\ad_{a\otimes b}\hatot  \id+\id \hatot  \ad_{a\otimes b})\sum_{p\in \Pi}\sum_\alpha (x_\alpha\otimes e_p)\otimes (y_\alpha \otimes f_p)\\
&=&-\sum_{i,j,\beta}\sum_\alpha\Big((a\circ x_\alpha\otimes b_{1i\beta})\otimes
(y_\alpha\otimes b_{2j\beta})+(a\circ x_\alpha\otimes b_{2j\beta})\otimes
(y_\alpha\otimes b_{1i\beta})
+(x_\alpha\circ a\otimes b_{2j\beta})\otimes (y_\alpha\otimes b_{1i\beta})\\
&&\hspace{0.4cm}+(x_\alpha\otimes b_{1i\beta})\otimes (a\circ y_\alpha\otimes b_{2j\beta})+(x_\alpha\otimes b_{2j\beta})\otimes (a\circ y_\alpha\otimes b_{1i\beta})+(x_\alpha\otimes b_{1i\beta})\otimes (y_\alpha\circ a\otimes b_{2j\beta})\Big)\\
&=& \delta(a\otimes b).
\end{eqnarray*}}
Hence the conclusion follows.
\end{proof}

Combining Theorem \ref{NYB-ND}, Proposition \ref{pro:NYBE}
and Corollary \ref{cor:same}, we obtain

\begin{thm}\label{const-Lie-bialgebra}
\label{thm:novlie}
Let $(A, \lhd, \rhd)$ be a \ndend algebra and $(A, \circ)$ be the
associated Novikov algebra.  Let $r$ be the skewsymmetric
solution of the \nybe in the Novikov algebra
$\widetilde{A}:= A\ltimes_{L_\rhd^\ast+R_\lhd^\ast, -R_\lhd^\ast}A^\ast$
defined by Eq.~\eqref{eq:solu}. Take $(B, \diamond, (\cdot,
\cdot))$ to be  a quadratic $\ZZ$-graded right Novikov algebra, and take the induced Lie algebra  $L$ to be $\widetilde{A}\otimes B$
 from $(\widetilde{A}, \circ)$ and $(B, \diamond)$. Then $r_{L}$ defined by Eq.~\eqref{eq:N-CYBE} is a skewsymmetric completed solution of the
CYBE in $L$. Hence,
there is a completed Lie bialgebra $(L,\delta)$ with $\delta$
defined by Eq.~\eqref{eq:rdelta} through $r_{L}$.
\vspc
\end{thm}

Now we introduce the Lie algebra structure corresponding to the ``affinization" of quasi-Frobenius Novikov algebras.
\vspb
\begin{defi}
Let $L=\oplus_{i\in \ZZ}L_i$ be a $\ZZ$-graded Lie algebra. If there is a skewsymmetric nondegenerate graded bilinear form $(\cdot, \cdot)_L$ on $L$ satisfying
\begin{eqnarray}
([a, b],c)_L+([c,a],b)_L+([b,c],a)_L=0\;\;\text{for all }  a, b, c\in L,
\end{eqnarray}
then $(L, (\cdot, \cdot)_L)$ is called a {\bf quasi-Frobenius $\ZZ$-graded Lie algebra}.
\vspb
\end{defi}
\begin{rmk}
For a quasi-Frobenius $\ZZ$-graded Lie algebra  $(L=\oplus_{i\in \ZZ}L_i, (\cdot, \cdot)_L)$, when $L=L_0$,  $(L, (\cdot, \cdot)_L)$ is the usual quasi-Frobenius Lie algebra \cite{BFS}.
\end{rmk}
\begin{pro}\label{Constr-quasi-Lie}
Let $(A, \circ)$ be a Novikov algebra, $(B=\oplus_{i\in \ZZ}B_i, \diamond, (\cdot,\cdot))$ be a quadratic $\ZZ$-graded right Novikov algebra, and $L=A\otimes B$ be the induced Lie algebra. Define a bilinear form $(\cdot, \cdot)_L$ on $L$ by
\vspb
\begin{eqnarray}\label{eq:quasi-Lie}
(a_1\otimes b_1, a_2\otimes
b_2)_L=\omega(a_1,a_2)(b_1,b_2)\;\;\;\text{for
all } a_1, a_2\in A, b_1, b_2\in B.
\end{eqnarray}
If $(A, \circ, \omega(\cdot,\cdot))$ is a quasi-Frobenius Novikov algebra, then $(L, (\cdot, \cdot)_L)$ is a quasi-Frobenius $\ZZ$-graded Lie algebra. Furthermore, if the quadratic $\mathbb{Z}$-graded right Novikov algebra is
$(B, \diamond, (\cdot,\cdot))$ $=({\bf k}[t,t^{-1}], \diamond,
(\cdot, \cdot))$ from Example~\ref{Laurent-Bilinear}, then $(L, (\cdot, \cdot)_L)$ is a quasi-Frobenius $\ZZ$-graded Lie algebra if and only if $(A, \circ, \omega(\cdot,\cdot))$ is a quasi-Frobenius Novikov algebra.
\end{pro}
\begin{proof}
 Since $\omega(\cdot,\cdot)$ is skewsymmetric and nondegenerate, and $(\cdot, \cdot)$ is symmetric, nondegenerate and graded, we obtain that $(\cdot, \cdot)_L$ is skewsymmetric, nondegenerate and graded. Let $a_1\otimes b_1$, $a_2\otimes b_2$ and $a_3\otimes b_3\in L$. We obtain
 \begin{eqnarray*}
 &&([a_1\otimes b_1, a_2\otimes b_2], a_3\otimes b_3)_L+([a_3\otimes b_3, a_1\otimes b_1], a_2\otimes b_2)_L+([a_2\otimes b_2, a_3\otimes b_3], a_1\otimes b_1)_L\\
 &=&\omega(a_1\circ a_2,a_3)(b_1\diamond b_2, b_3)-\omega(a_2\circ a_1,a_3)(b_2\diamond b_1, b_3)
 +\omega(a_3\circ a_1,a_2)(b_3\diamond b_1, b_2)\\
 &&-\omega(a_1\circ a_3,a_2)(b_1\diamond b_3, b_2)+\omega(a_2\circ a_3,a_1)(b_2\diamond b_3, b_1)-\omega(a_3\circ a_2,a_1)(b_3\diamond b_2, b_1)\\
 &=& (\omega(a_1\circ a_2,a_3)-\omega(a_1\circ a_3,a_2))(b_1\diamond b_2, b_3)
 +(\omega(a_2\circ a_3,a_1)-\omega(a_2\circ a_1,a_3))(b_2\diamond b_1, b_3)\\
 &&+(\omega(a_3\circ a_1,a_2)-\omega(a_3\circ a_2,a_1))(b_3\diamond b_1, b_2)\\
 &=&(\omega(a_1\circ a_2,a_3)-\omega(a_1\circ a_3,a_2))(b_1\diamond b_2, b_3)
 +(\omega(a_2\circ a_3,a_1)-\omega(a_2\circ a_1,a_3))(b_2\diamond b_1, b_3)\\
 &&-(\omega(a_3\circ a_1,a_2)-\omega(a_3\circ a_2,a_1))(b_1\diamond b_2+b_2\diamond b_1, b_3)\\
 &=&(\omega(a_1\circ a_2, a_3)-\omega(a_1\star a_3, a_2)+\omega(a_3\circ a_2, a_1))(b_1\diamond b_2, b_3)\\
 &&-(\omega(a_2\circ a_1, a_3)-\omega(a_2\star a_3, a_1)+\omega(a_3\circ a_1, a_2))(b_2\diamond b_1,b_3)\\
 &=& 0.
 \end{eqnarray*}
 Therefore, $(L, (\cdot, \cdot)_L)$ is a quasi-Frobenius $\ZZ$-graded Lie algebra.
Suppose that the quadratic $\mathbb{Z}$-graded right Novikov algebra is
$(B, \diamond, (\cdot,\cdot))$ $=({\bf k}[t,t^{-1}], \diamond,
(\cdot, \cdot))$ from Example~\ref{Laurent-Bilinear}  and $(L, (\cdot, \cdot)_L)$ is a quasi-Frobenius $\ZZ$-graded Lie algebra. It is easy to see that $\omega(\cdot,\cdot)$ is skewsymmetric. For all $a$, $b$, $c\in A$ and $m$, $n$, $k\in \ZZ$, we have
\begin{eqnarray*}
0&=&([at^m,bt^n],ct^k)_L+([ct^k, at^m], bt^n)_L+([bt^m, ct^k], at^m)_L\\
&=& m(\omega(a\circ b, c)-\omega(a\star c,b)+\omega(c\circ b, a))-n(\omega(b\circ a, c)-\omega(b\star c, a)+\omega(c\circ a,b)).
\end{eqnarray*}
Setting $m=1$ and $n=0$ yields that $(A, \circ, \omega(\cdot,\cdot))$ is a quasi-Frobenius Novikov algebra. Then the proof is completed.
\vspb
\end{proof}
Combining Propositions \ref{quasi-Nov-equi}, \ref{pro:NYBE} and \ref{Constr-quasi-Lie}, we obtain the equivalences.
\vspb
\begin{thm}\label{quasi-Lie-equi}
Let $(A, \circ)$ be a Novikov algebra with a nondegenerate bilinear form $\omega(\cdot,\cdot)$. Let $\{e_\alpha|\alpha\in I\}$ be a basis of $A$ and $\{f_\alpha|\alpha\in I\}$ be its dual basis associated with $\omega(\cdot,\cdot)$. Set $r=\sum_{\alpha \in I}e_\alpha \otimes f_\alpha \in A\otimes A$ and let $L=A[t,t^{-1}]$ be the affinization of $(A, \circ)$. Then the following conditions are equivalent.
\begin{enumerate}
\item $(A, \circ, \omega(\cdot,\cdot))$ is a quasi-Frobenius Novikov algebra.
\item $r$ is a skewsymmetric solution of the NYBE in $A$.
\item \label{it:fourc} $r_L=\sum_{ i\in
\mathbb{Z}}\sum_{\alpha\in I} e_\alpha t^i\,\otimes \,f_\alpha t^{-i-1}\in L
\hatot  L$ is a skewsymmetric completed solution of the CYBE in $L$.
\item \label{it:fourd} $(L, (\cdot, \cdot)_L)$ is a quasi-Frobenius $\ZZ$-graded Lie algebra with $(\cdot, \cdot)_L$ defined by $(at^i,bt^j)_L=\omega(a,b)\delta_{i+j+1,0}$ for all $a$, $b\in A$ and $i$, $j\in \ZZ$.
        \end{enumerate}
\vspc
\end{thm}
It is worth remarking that the equivalence of $\eqref{it:fourc}$ and $\eqref{it:fourd}$ provides instances in which the existing equivalence~\cite{BFS} between skewsymmetric solutions of the CYBE and quasi-Frobenius Lie algebras for finite-dimensional Lie algebras is extended to $\ZZ$-graded Lie algebras.

By Theorems \ref{quasi-Lie-equi} and \ref{NYB-ND}, we obtain the following consequence.
\begin{cor}
Let $(A, \lhd, \rhd)$ be a \ndend algebra and $(A, \circ)$ be the
associated Novikov algebra.  Let
$\widetilde{A}:= A\ltimes_{L_\rhd^\ast+R_\lhd^\ast, -R_\lhd^\ast}A^\ast$ be the semi-direct product Novikov algebra.
 Take $(B, \diamond, (\cdot,
\cdot))$ to be  a quadratic $\ZZ$-graded right Novikov algebra, and take the induced Lie algebra  $L$ to be $\widetilde{A}\otimes B$
 from $(\widetilde{A}, \circ)$ and $(B, \diamond)$. Then $(L, (\cdot, \cdot)_L)$ with $(\cdot, \cdot)_L$ defined by Eq.~\eqref{eq:quasi-Lie} is a quasi-Frobenius Lie algebra, where $\omega(\cdot,\cdot)$ is given by Eq.~\eqref{skewsymm-bin}.
\end{cor}

To finish the paper, we present a simple example of completed Lie bialgebras and quasi-Frobenius $\ZZ$-graded Lie algebras obtained by pre-Novikov algebras.

\begin{ex} \label{final-ex} Let $(A={\bf k}e, \lhd, \rhd)$ be the one-dimensional pre-Novikov algebra given by
    \vspa
\begin{eqnarray*}
e \lhd e=e,~~~~e\rhd e=0.
\end{eqnarray*}
Let $e^\ast \in A^\ast$ be the dual basis. Then the
Novikov algebra $A\ltimes_{L_\rhd^\ast+R_\lhd^\ast,
-R_{\lhd}^\ast} A^\ast$ is the vector space ${\bf k}e\oplus {\bf k}e^\ast$
endowed with the multiplication
\vspa
\begin{eqnarray} \notag
e\circ e=e, ~~~e\circ e^\ast=-e^\ast,~~~e^\ast\circ
e=e^\ast,~~~~e^\ast\circ e^\ast=0.
\end{eqnarray}
By Theorem \ref{NYB-ND}, $r=e\otimes e^\ast-e^\ast\otimes e$ is a
skewsymmetric solution of the \nybe in
$A\ltimes_{L_\rhd^\ast+R_\lhd^\ast, -R_{\lhd}^\ast} A^\ast$.

Let $(B, \diamond)=({\bf k}[t,t^{-1}], \diamond, (\cdot,
\cdot))$ be the quadratic $\mathbb{Z}$-graded right Novikov
algebra given in Example \ref{Laurent-Bilinear}. The induced
infinite-dimensional Lie algebra $L$ from
$A\ltimes_{L_\rhd^\ast+R_\lhd^\ast, -R_{\lhd}^\ast} A^\ast$ and
$(B, \diamond)$ is the vector space spanned by $\{e_i:=et^i, f^i:=e^\ast
t^{-i-1}\,|\, i\in \mathbb{Z}\}$ with the Lie brackets
\begin{eqnarray*}
[e_i, e_j]=(i-j)e_{i+j-1},~~~~[e_i, f^j]=-(i-j-1)f^{j-i+1},~~~[f^i, f^j]=0 ~~ \text{ for all } i, j\in \mathbb{Z}.
\end{eqnarray*}
By Proposition \ref{pro:NYBE}, $r_L=\sum_{i\in
\mathbb{Z}}(e_i\otimes f^i-f^i\otimes
e_i)$ is a skewsymmetric completed solution of the
CYBE in $L$. Then by
Corollary~\ref{cor:same}, there is a \complete Lie bialgebra
structure $(L, [\cdot, \cdot], \delta)$ given by
$$\delta(e_k)=k\Big(\sum_{i\in \mathbb{Z}}e_{k+i-1}\otimes
f^i-f^i\otimes
e_{k+i-1}\Big),\;\delta(f^k)=\sum_{i\in
\mathbb{Z}}(-k+2i-1) f^{k-i+1}\otimes f^i \  \text{ for all } k\in \ZZ.
$$

By Theorem \ref{NYB-ND}, $(A\ltimes_{L_\rhd^\ast+R_\lhd^\ast,
-R_{\lhd}^\ast} A^\ast, \bullet, \omega(\cdot,\cdot))$ is a quasi-Frobenius Novikov algebra with the bilinear form $\omega(\cdot,\cdot)$ given by
\begin{eqnarray*}
\omega(e,e)=\omega(e^\ast,e^\ast)=0,\;\;\omega(e,e^\ast)=-\omega(e^\ast,e)=1.
\end{eqnarray*}
Then by Theorem \ref{quasi-Lie-equi}, $(L, (\cdot,\cdot)_L)$ is a quasi-Frobenius Lie algebra with the bilinear form $(\cdot,\cdot)_L$ given by
\begin{eqnarray*}
(e_i, e_j)_L=(f^i, f^j)_L=0, (e_i, f^j)_L=-(f^j,
e_i)_L=\delta_{i,j}\;\;\text{for all } i, j\in \ZZ.
\end{eqnarray*}
Note that $(L, (\cdot, \cdot)_L)$  is a Frobenius Lie algebra, that is, there exists a linear function $F$ on $L$ such that $F([x,y])=(x,y)_L$ for all $x$, $y\in L$. In fact, the linear function $F$ on $L$ is defined by
\begin{eqnarray*}
F(e_i)=0,\;\; F(f^i)=\delta_{i,1} \;\;\text{for all $i\in \ZZ$.}
\end{eqnarray*}
One can check that  $F([x,y])=(x,y)_L$ for all $x$, $y\in L$. Therefore, $(L, (\cdot, \cdot)_L)$  is a Frobenius Lie algebra.
\vspb
\end{ex}

\noindent {\bf Acknowledgments.} This research is supported by
NSFC (12171129, 11871421, 11931009, 12271265),
the Zhejiang Provincial Natural Science Foundation of China (LY20A010022) and the Scientific Research Foundation of Hangzhou Normal University (2019QDL012), the Fundamental Research Funds for the Central Universities and Nankai Zhide Foundation. We thank Jun Pei for helpful discussions. We give our appreciation to the referees for their suggestions that have improved the paper.

\smallskip

\noindent
{\bf Declaration of interests. } The authors have no conflicts of interest to disclose.

\smallskip

\noindent
{\bf Data availability. } No new data were created or analyzed in this study.

\end{document}